\documentclass[12pt,A4]{amsart}
\usepackage{amsfonts,amsthm,amsmath}
\usepackage{rotating}
\usepackage{tikz}
\usepackage{verbatim}
\usepackage{amscd}
\usepackage[latin2]{inputenc}
\usepackage{t1enc}
\usepackage[mathscr]{eucal}
\usepackage{indentfirst}
\usepackage{graphicx}
\usepackage{graphics}
\usepackage{pict2e}
\usepackage{mathrsfs}
\usepackage{enumerate}
\usepackage{amssymb,latexsym}

\usepackage{xparse}

\usepackage[pagebackref]{hyperref}
\hypersetup{backref, pagebackref, colorlinks=true}
\usepackage{cite}
\usepackage{color}
\usepackage{epic}
\usepackage{hyperref} 
\usepackage{tkz-graph}

\input pathlate.sty

\numberwithin{equation}{section}
\theoremstyle{plain}
\newtheorem{theorem}{Theorem}[section]
\newtheorem{lemma}[theorem]{Lemma}
\newtheorem{corollary}[theorem]{Corollary}

\newtheorem{claim}[theorem]{Claim}
\newtheorem{proposition}[theorem]{Proposition}

\theoremstyle{definition}

\newtheorem{remark}[theorem]{Remark}
\newtheorem{?}[theorem]{Problem}

\NewDocumentCommand{\slanttext}{O{}m}{\rotatebox[origin=c]{45}{#2}}

\def\boxit#1{\leavevmode\hbox{\vrule\vtop{\vbox{\kern.33333pt\hrule\kern1pt\hbox{\kern1pt\vbox{#1}\kern1pt}}\kern1pt\hrule}\vrule}}

\usepackage{collectbox}


\begin{document}

\title[Hecke growth diagrams]{Hecke growth diagrams, and maximal increasing and decreasing sequences in fillings of stack polyominoes}

\author[T. Guo and G. Li]{Ting Guo and Gaofan Li}
\address[T. Guo \& G. Li]{MOE-LCSM, School of Mathematics and Statistics, Hunan Normal University, China.} 
\email{guoting@hunnu.edu.cn(T. Guo), 1636517624@qq.com(G. Li).}

\date{\today}

\begin{abstract} We establish a bijection between $01$-fillings of stack polyominoes with at most one $1$ per column and labelings of the corners along the top-right border of stack polyominoes. 
These labellings indicate the lengths of the longest increasing and decreasing chains of the largest rectangular region below and to the left of the corners.
Our results provide an alternative proof of Guo and Poznanovi\'c's theorem on the lengths of the longest increasing and decreasing chains have a symmetric joint distribution over $01$-fillings of stack polyomino.
Moreover, our results offer new perspective to Chen, Guo and Pang's result on the crossing number and the nesting number have a symmetric joint distribution over linked partitions.
In particular, our construction generalizes the growth diagram techniques of Rubey for the $01$-fillings of stack polyominoes with at most one $1$ per column and row. 

\end{abstract} 

\keywords{Hecke growth diagrams, Jeu de taquin, Maximal chains, Stack polyominoes}

\maketitle


\section{Introduction}

A $polyomino$ is a finite subset of $\mathbb{Z}^2$, where every element of $\mathbb{Z}^2$ is regarded as a square box. A {\it column} (resp. {\it row}) of a polyomino is the set of boxes along a vertical(resp. horizontal) line.
A polyomino is {\it row-convex} (resp. {\it column-convex}) if for any two boxes in a row (resp. column), the elements of $\mathbb{Z}^2$ between them are also boxes belong to the polyomino. If the polyomino is both row-convex and column-convex, we say that it is {\it convex}.
A polyomino is {\it intersection-free\/} if every two columns are {\it comparable\/}, i.e., the row-coordinates of one column form a subset of those of the other column. 
Equivalently, it is intersection-free if every two rows are comparable.

A polyomino is a {\it moon polyomino} if it is convex and intersection-free (e.g. Fig.\ref{fig:1}a).  We say that a moon polyomino is a {\it stack polyomino} if its rows are left justified (e.g. Fig.\ref{fig:1}b). We note that the term "stack polyomino" has been used in other papers to denote polyominoes with justified columns rather than rows. A {\it Ferrers diagram} (also called {\it Ferrers shape} or {\it Young diagram}) is a stack polyomino with weakly decreasing row widths, reading rows from top to bottom(e.g. Fig.\ref{fig:1}c).


\begin{figure}[h]
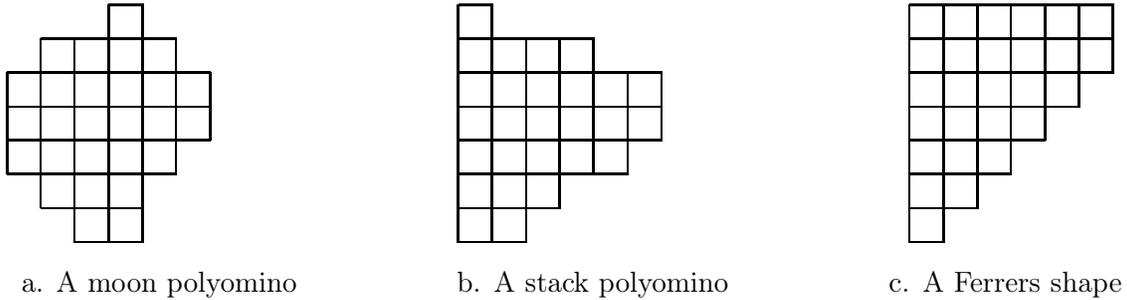

	$$
	\Einheit.15cm
	\PfadDicke{.5pt}
	\Pfad(0,6),222222222\endPfad
	\Pfad(3,3),222222222222222\endPfad
	\Pfad(6,0),222222222222222222\endPfad
	\Pfad(9,0),222222222222222222222\endPfad
	\Pfad(12,0),222222222222222222222\endPfad
	\Pfad(15,6),222222222222\endPfad
	\Pfad(18,9),222222\endPfad
	\Pfad(9,21),111\endPfad
	\Pfad(3,18),111111111111\endPfad
	\Pfad(0,15),111111111111111111\endPfad
	\Pfad(0,12),111111111111111111\endPfad
	\Pfad(0,9),111111111111111111\endPfad
	\Pfad(0,6),111111111111111\endPfad
	\Pfad(3,3),111111111\endPfad
	\Pfad(6,0),111111\endPfad
	\hbox{\hskip6cm}
	\PfadDicke{.5pt}
	\Pfad(0,0),111111\endPfad
	\Pfad(0,3),111111111\endPfad
	\Pfad(0,6),111111111111111\endPfad
	\Pfad(0,9),111111111111111111\endPfad
	\Pfad(0,12),111111111111111111\endPfad
	\Pfad(0,15),111111111111111111\endPfad
	\Pfad(0,18),111111111111\endPfad
	\Pfad(0,21),111\endPfad
	\Pfad(0,0),222222222222222222222\endPfad
	\Pfad(3,0),222222222222222222222\endPfad
	\Pfad(6,0),222222222222222222\endPfad
	\Pfad(9,3),222222222222222\endPfad
	\Pfad(12,6),222222222222\endPfad
	\Pfad(15,6),222222222\endPfad
	\Pfad(18,9),222222\endPfad
	\hbox{\hskip6cm}
	\PfadDicke{0.5pt}
	\Pfad(0,0),222222222222222222222\endPfad
	\Pfad(3,0),222222222222222222222\endPfad
	\Pfad(6,3),222222222222222222\endPfad
	\Pfad(9,6),222222222222222\endPfad
	\Pfad(12,9),222222222222\endPfad
	\Pfad(15,12),222222222\endPfad
	\Pfad(18,15),222222\endPfad
	\Pfad(0,0),111\endPfad
	\Pfad(0,3),111111\endPfad
	\Pfad(0,6),111111111\endPfad
	\Pfad(0,9),111111111111\endPfad
	\Pfad(0,12),111111111111111\endPfad
	\Pfad(0,15),111111111111111111\endPfad
	\Pfad(0,18),111111111111111111\endPfad
	\Pfad(0,21),111111111111111111\endPfad
	\hskip3cm
	$$
	\centerline{\small a. A moon polyomino
		\hskip2cm
		b. A stack polyomino
		\hskip2cm
		c. A Ferrers shape}
	\caption{A moon polyomino, a stack polyomino and a Ferrers shape.}
	\label{fig:1}
\end{figure}

In Fig.\ref{fig:2}, the polyomino is not a moon polyomino because the first column (from the left) and the last column are not intersection-free.

\begin{figure}
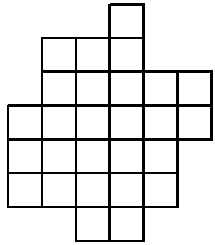

	$$
	\Einheit.15cm
	\PfadDicke{.5pt}
	\Pfad(0,3),222222222\endPfad
	\Pfad(3,3),222222222222222\endPfad
	\Pfad(6,0),222222222222222222\endPfad
	\Pfad(9,0),222222222222222222222\endPfad
	\Pfad(12,0),222222222222222222222\endPfad
	\Pfad(15,3),222222222222\endPfad
	\Pfad(18,9),222222\endPfad
	\Pfad(9,21),111\endPfad
	\Pfad(3,18),111111111\endPfad
	\Pfad(3,15),111111111111111\endPfad
	\Pfad(0,12),111111111111111111\endPfad
	\Pfad(0,9),111111111111111111\endPfad
	\Pfad(0,6),111111111111111\endPfad
	\Pfad(0,3),111111111111111\endPfad
	\Pfad(6,0),111111\endPfad
	\hskip3cm
	$$
	\caption{A polyomino but not a moon polyomino.}
	\label{fig:2}
\end{figure}

A {\it $01$-filling} of a moon polyomino is an assignment of either a $0$ or a $1$ to each box of the polyomino.
Usually, we adjust the convention for better visibility by representing the filling with suppressed $0$'s and replacing $1$'s by $X$'s.

A $partition$ is a weakly decreasing sequence $\lambda=(\lambda_{1},\lambda_{2},...,\lambda_{n})$ of positive integers. Supposing there are two partitions $\mu=(\mu_{1},\mu_{2},...,\mu_{l})$ and $\upsilon=(\upsilon_{1},\upsilon_{2},...,\upsilon_{m})$, $l\leq m$, then $\mu\subseteq\upsilon$ if and only if $\mu_{i}\leq\upsilon_{i}$ for all $i\leq l$.

To each partition $\lambda=(\lambda_{1},\lambda_{2},...,\lambda_{n})$, one associates its Ferrers shape with $\lambda_{i}$ boxes in the $i$th row for $i=1,2,...,n$. For example, the partition $\lambda=(6,6,5,4,3,2,1)$ corresponds to the Ferrers shape of Fig.\ref{fig:1}c. By the correspondence between partitions and Ferrers shapes, one can define the transpose of a partition.

The $conjugate$ of a partition $\lambda=(\lambda_{1},\lambda_{2},...,\lambda_{n})$ is the partition $\lambda^{'}=(\lambda_{1}^{'},\lambda_{2}^{'},...,\lambda_{\lambda_{1}}^{'})$, where $\lambda_{i}^{'}$ is the length of the {\it i-th} column in the Ferrers shape of $\lambda$.
For instance, the conjugate of partition $\lambda=(4,2,2,1)$ is the partition $\lambda^{'}=(4,3,1,1)$. 


The main focus of this paper is on constructing bijection between $01$-fillings of stack polyominoes with at most one $1$ per column and labelings of the corners along the top-right border of stack polyominoes, and investigating the increasing and decreasing chains in fillings of stack polyominoes.

The application of growth diagram for polyominoes has been extensively studied in prior research. In \cite{kra}, Krattenthaler used Fomin's growth diagrams to establish a bijection between $01$-fillings of Ferrers shape $F$, where each row and column contains at most one $1$, and sequences of partitions along the top-right border of $F$. This correspondence provided an alternative proof of the main results from \cite{CDDSY}
concerning crossings and nestings of matching and set partitions. Krattenthaler also extended this correspondence to arbitrary fillings of Ferrers shapes.
His motivation for proving these extensions was the result of Jonsson \cite{Jonsson} and Jonsson and Welker \cite{JW}, which state that the number of $01$-fillings of a stack polyomino $M$ with longest northeast chain of length $u$ depends only on the set of columns of $M$. (Rubey further extended Jonsson and Welker's results to moon polyominoes using a bijection based on the RSK insertion, jeu de taquin, and the inclusion and exclusion principle. In \cite{PY}, Poznanovi\'c and Yan given a completely bijective proof and extended these results to almost-moon polyominoes.)

In \cite{Rubey}, Rubey extended the growth diagram of Ferrers shapes to stack polyominoes. We note that, in Rubey's work, stack polyominoes refer to moon polyominoes where columns are bottom-justified. He used Fomin's growth diagrams and jeu de taquin to establish a bijection between $01$-fillings of stack polyomino $\mathcal{M}$, where each row and column contains at most one $1$, and labellings of corners along the top-right border of $\mathcal{M}$.

In this paper, we further extend Rubey's result to words, providing a bijection between 01-fillings of stack polyomino $\mathcal{M}$, where each column contains at most one $1$, and labellings of corners along the top-right border of $\mathcal{M}$. 
We show that these labellings indicate the lengths of the longest increasing and decreasing chains of the largest rectangular region below and to the left of the corners.

This paper is organized as follows: In Section 2, we provide the necessary background. In Section 3, we present the backward local rules of Hecke growth diagrames for the $01$-fillings of rectangulars. Section 4 contains the two main results of this paper. Section 5 shows the two alternative proofs of the results by Guo and Poznanovi\'c \cite{GP2020}, Chen, Guo and Pang \cite{CGP}, and we demonstrate our Hecke growth diagram generalizes Rubey's approach. Section 6 discusses an alternative labeling along the top-right border of stack polyominoes with a bottom-justified border.


\section{Preliminaries}

To establish the bijection that we mentioned in Section 1, we need several foundational tools: Hecke insertion \cite{BKSTY}, Hecke growth diagrams \cite{PP}, jeu de taquin for increasing tableaux \cite{TY09}, 
and the $K$-Knuth equivalence on words and increasing tableaux \cite{BuchSamuel,TY11}. For the convenience of readers,
we provide detailed description of these definitions in this section, which are also available in \cite{GP2020}.

\subsection{Hecke insertion}
$\newline$

Hecke insertion, introduced in \cite{BKSTY}, was developed for studying the stable Grothendieck polynomials. In this section, we follow \cite{BKSTY} to give a description of Hecke insertion.


A tableau of shape $\rho/\lambda$ is called an {\it increasing tableau} if the entries are strictly increase along each row and down each column. Let INC($\rho/\lambda$) be the set of these increasing tableaux of skew shape $\rho/\lambda$. If the set of entries constitute an integer interval with left endpoint 1, we call this increasing tableau a {\it standard increasing tableau}. The set of these standard increasing tableaux of skew shape $\rho/\lambda$ is denoted by INC$^{st}$($\rho/\lambda$).

%

An increasing tableau is called a {\it straight increasing tableau} if its shape forms a straight northwest border. The set of these straight increasing tableaux of shape $\lambda$ is denoted by SIT($\lambda$). If the entries constitute a consecutive integer sequence starting from $1$, we call this straight increasing tableau a {\it standard straight increasing tableau}, denoted by SIT$^{st}$($\lambda$).

The Hecke insertion is a procedure to insert a positive integer $x$ into a straight increasing tableau $Y$,
resulting in another straight increasing tableau $Z$, which either $Z$ has the same shape as $Y$ 
or it has one extra box $c$, we define the extra box $c$ the {\it terminate box}.
In the case when $Z$ has the same shape as $Y$, it also contains a special corner $c$, i.e. terminate box $c$, where the insertion algorithm terminated. We use a parameter $\alpha\in \lbrace 0, 1 \rbrace$ to distinguish these two case, 
which $\alpha$ is set to $1$ if and only if the corner $c$ is outside the shape of $Y$. 
Thus the complete output of the insertion algorithm is the triple $(Z, c, \alpha)$.
We write $(Z, c, \alpha)= (Y\xleftarrow{\,\mathrm{H}}x)$.

The Hecke insertion algorithm proceeds by inserting the integer $x$ into the first row of $Y$. 
This may modify this row, and possibly produce an {\it output integer}, 
which is then inserted into the second row of $Y$, etc. 
This process is repeated until an insertion does not produce an output integer. 
The rules for inserting an integer $x$ into a row $R$ is as follows.

If $x$ is larger than or equal to all entries in $R$, then no output integer is produced and the algorithm terminates.
If adding $x$ as a new box to the end of $R$ results in an increasing tableau, 
then $Z$ is the resulting tableau, $\alpha=1$, and terminate box $c$ is the box where $x$ was added.
If adding $x$ as a new box to the end of $R$ does not result in an increasing tableau, 
then let $Z=Y$, $\alpha=0$, and terminate box $c$ is the box at the bottom of the column of $Z$ 
containing the rightmost box of $R$.

Otherwise the integer $x$ is strictly smaller than some entry in $R$, 
and we let $y$ be the smallest integer in $R$ that is strictly larger than $x$.
If replacing $y$ with $x$ results in an increasing tableau, then replace $y$ with $x$ and insert $y$ into the next row.
If replacing $y$ with $x$ does not result in an increasing tableau, 
then insert $y$ into the next row and do not change $R$.

\begin{figure}
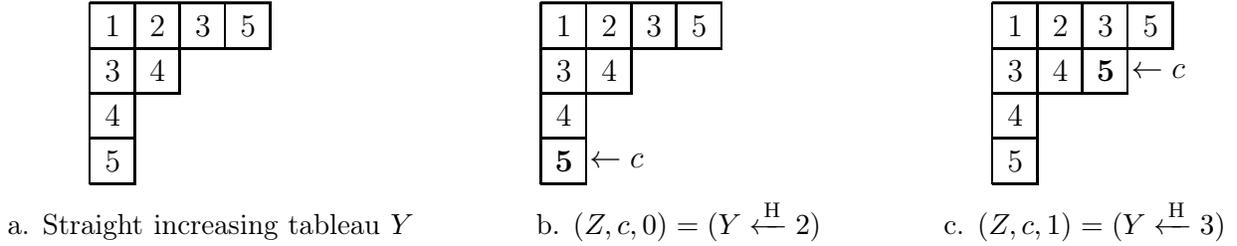

	$$
	\Einheit.2cm
	\PfadDicke{.5pt}
	\Pfad(0,0),222222222222\endPfad
	\Pfad(3,0),222222222222\endPfad
	\Pfad(6,6),222222\endPfad
	\Pfad(9,9),222\endPfad
	\Pfad(12,9),222\endPfad
	\Pfad(0,0),111\endPfad
	\Pfad(0,3),111\endPfad
	\Pfad(0,6),111111\endPfad
	\Pfad(0,9),111111111111\endPfad
	\Pfad(0,12),111111111111\endPfad
	\Label\ro{1}(1,10)
	\Label\ro{2}(4,10)
	\Label\ro{3}(7,10)
	\Label\ro{5}(10,10)
	\Label\ro{3}(1,7)
	\Label\ro{4}(4,7)
	\Label\ro{4}(1,4)
	\Label\ro{5}(1,1)
	\hbox{\hskip6cm}
	\Einheit.2cm
	\PfadDicke{.5pt}
	\Pfad(0,0),222222222222\endPfad
	\Pfad(3,0),222222222222\endPfad
	\Pfad(6,6),222222\endPfad
	\Pfad(9,9),222\endPfad
	\Pfad(12,9),222\endPfad
	\Pfad(0,0),111\endPfad
	\Pfad(0,3),111\endPfad
	\Pfad(0,6),111111\endPfad
	\Pfad(0,9),111111111111\endPfad
	\Pfad(0,12),111111111111\endPfad
	\Label\ro{1}(1,10)
	\Label\ro{2}(4,10)
	\Label\ro{3}(7,10)
	\Label\ro{5}(10,10)
	\Label\ro{3}(1,7)
	\Label\ro{4}(4,7)
	\Label\ro{4}(1,4)
	\Label\ro{\textbf 5}(1,1)
	\Label\ro{\leftarrow c}(4.5,1)
	\hbox{\hskip6cm}
	\Einheit.2cm
	\PfadDicke{.5pt}
	\Pfad(0,0),222222222222\endPfad
	\Pfad(3,0),222222222222\endPfad
	\Pfad(6,6),222222\endPfad
	\Pfad(9,6),222222\endPfad
	\Pfad(12,9),222\endPfad
	\Pfad(0,0),111\endPfad
	\Pfad(0,3),111\endPfad
	\Pfad(0,6),111111111\endPfad
	\Pfad(0,9),111111111111\endPfad
	\Pfad(0,12),111111111111\endPfad
	\Label\ro{1}(1,10)
	\Label\ro{2}(4,10)
	\Label\ro{3}(7,10)
	\Label\ro{5}(10,10)
	\Label\ro{3}(1,7)
	\Label\ro{4}(4,7)
	\Label\ro{\textbf 5}(7,7)
	\Label\ro{4}(1,4)
	\Label\ro{5}(1,1)
	\Label\ro{\leftarrow c}(10.5,7)
	\hskip2cm
	$$
	\centerline{\small a. Straight increasing tableau $Y$
		\hskip1.5cm
		b. $(Z,c,0)=(Y\xleftarrow{\,\mathrm{H}}2)$
		\hskip1.5cm
		c. $(Z,c,1)=(Y\xleftarrow{\,\mathrm{H}}3)$}
	\caption{Two examples of Hecke insertion.}
	\label{two hecke}
\end{figure}

For example, let $Y$ be a straight increasing tableau given in Fig.\ref{two hecke}a. Suppose we wish to compute $(Y\xleftarrow{\,\mathrm{H}}2)$. Since the first row of $Y$ contains $2$, $3$ is inserted into the second row and the first row remains unchanged. Since the second row of $Y$ contains $3$, $4$ is inserted into the third row and the second row remains unchanged. Since the value of the rightmost box in the third row is $4$, the algorithm terminates with $\alpha=0$, and $c=(4,1)$ is the box in the fourth row and first column. The result tableau $Z$ is given in Fig.\ref{two hecke}b. 

On the other hand, suppose we wish to compute $(Y\xleftarrow{\,\mathrm{H}}3)$. $5$ is inserted into the second row and the first row remains unchanged. Since $5$ is larger than all entries in the second row of $Y$, $5$ is added as a new box to the end of the second row of $Y$. The algorithm terminates with the tableau $Z$ in Fig.\ref{two hecke}c, $\alpha=1$, and $c=(2,3)$.

Hecke insertion algorithm is reversible. The reverse Hecke insertion algorithm is given in \cite{BKSTY}.
Hecke insertion is a generalization of the standard Robinson-Schensted-Knuth (RSK) algorithm.
While the RSK algorithm associates a permutation with a pair of standard Young tableaux,
the Hecke insertion establishes a correspondence between words and $(P,\,Q)$,
where $P$ is a straight increasing tableau and $Q$ is a set-valued tableau.
We call $P$ the {\it Hecke insertion tableau} and $Q$ the {\it Hecke recording tableau}.
{\it Set-valued tableaux} were introduced by Buch \cite{Buch} in the study of the $K$-theory of Gra{\ss}mannians,
which are $\mathbb{N}$-fillings of a Young diagram assigning to each box a nonempty set of positive integers such that
the largest entry of a box is smaller than the smallest entry in the boxes directly to the right of it, 
and directly below it.
Let $w=w_{1}\cdots w_{n}$ be a word. The corresponding pair $(P(w),Q(w))$ is constructed as follows.

Hecke inserting the letters of $w$ recursively from $w_{1}$ to $w_{n}$, we construct the Hecke insertion tableau:
$$P(w)=(\cdots((\emptyset\xleftarrow{\,\mathrm{H}}w_{1})\xleftarrow{\,\mathrm{H}}w_{2})\cdots\xleftarrow{\,\mathrm{H}}w_{n}).$$

Set $Q(\emptyset)=\emptyset$, $P(w_{1}\cdots w_{k})$ is obtained from $P(w_{1}\cdots w_{k-1})$, while $Q(w_{1}\cdots w_{k})$ is obtained from $Q(w_{1}\cdots w_{k-1})$ by placing the integer $k$ into a special box $c'$ of $Q(w_{1}\cdots w_{k})$. The position of the special box $c'$ is the same as the position of the terminating box $c$ of $P(w_{1}\cdots w_{k})$.
For example, let $w=5433124235$, $P(w)$ and $Q(w)$ are given in Fig.\ref{(P,Q) for word}.

\begin{figure}[h]
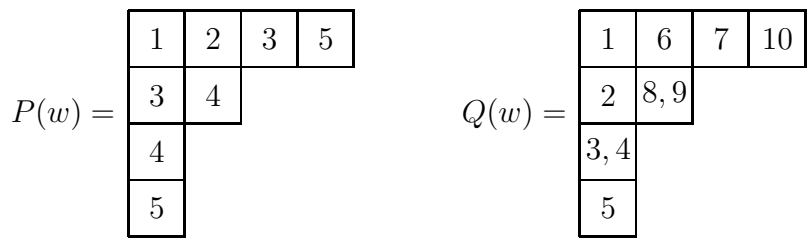

	$$
	\Einheit.25cm
	\PfadDicke{.5pt}
	\Pfad(0,0),222222222222\endPfad
	\Pfad(3,0),222222222222\endPfad
	\Pfad(6,6),222222\endPfad
	\Pfad(9,9),222\endPfad
	\Pfad(12,9),222\endPfad
	\Pfad(0,0),111\endPfad
	\Pfad(0,3),111\endPfad
	\Pfad(0,6),111111\endPfad
	\Pfad(0,9),111111111111\endPfad
	\Pfad(0,12),111111111111\endPfad
	\Label\ro{1}(1,10)
	\Label\ro{2}(4,10)
	\Label\ro{3}(7,10)
	\Label\ro{5}(10,10)
	\Label\ro{3}(1,7)
	\Label\ro{4}(4,7)
	\Label\ro{4}(1,4)
	\Label\ro{5}(1,1)
	\Label\ro{P(w)=}(-4,6)
	\hbox{\hskip6cm}
	\Einheit.25cm
	\PfadDicke{.5pt}
	\Pfad(0,0),222222222222\endPfad
	\Pfad(3,0),222222222222\endPfad
	\Pfad(6,6),222222\endPfad
	\Pfad(9,9),222\endPfad
	\Pfad(12,9),222\endPfad
	\Pfad(0,0),111\endPfad
	\Pfad(0,3),111\endPfad
	\Pfad(0,6),111111\endPfad
	\Pfad(0,9),111111111111\endPfad
	\Pfad(0,12),111111111111\endPfad
	\Label\ro{1}(1,10)
	\Label\ro{6}(4,10)
	\Label\ro{7}(7,10)
	\Label\ro{10}(10,10)
	\Label\ro{2}(1,7)
	\Label\ro{8,9}(4,7)
	\Label\ro{3,4}(1,4)
	\Label\ro{5}(1,1)
	\Label\ro{Q(w)=}(-4,6)
	\hskip2cm
	$$
	\caption{The pair $(P(w),Q(w))$ of an increasing and set-valued tableaux for $w=5433124235$.}
	\label{(P,Q) for word}
\end{figure}

For a word $w$, we denote the length of the longest strictly increasing (resp. decreasing) subsequence of $w$ by $lis(w)$ (resp. $lds(w)$). For a straight increasing tableau $P$, we denote the number of columns (resp. rows) of $P$ by $c(P)$ (resp. $r(P)$).
In \cite{TY11}, Thomas and Yong proved that the insertion tableau of a word formed by Hecke insertion determines the lengths of the longest strictly increasing and strictly decreasing subsequences in the word. That is the following theorem in \cite{TY11}. 

\begin{theorem}
	Let $w$ be a word and let $P(w)$ be the Hecke insertion tableau of $w$. Then $lis(w)=c(P)$ and $lds(w)=r(P)$.
	\label{Th2.1}
\end{theorem}

\subsection{Hecke growth diagrams for $01$-fillings of rectangular}
$\newline$

The growth diagrams for Hecke insertion were given by Patrias and Pylyavskyy in \cite{PP}. Following \cite{PP}, we describe the Hecke growth diagram of a word.

We represent a word $w=w_{1}w_{2}\cdots$ containing $a_{1}$ copies of 1, $a_{2}$ copies of 2, ..., $a_{m}$ copies of $m$ as a filling of an $m\times (a_{1}+a_{2}+\cdots +a_{m})$ rectangle, by placing an $X$ into the $i$-th column and $w_i$-th row (where we number the rows from bottom to top) for $i=1,2,...,a_{1}+a_{2}+\cdots +a_{m}$. An example is shown in Fig.\ref{fig:word}. 

\begin{figure}[h]
	$$
	\Einheit0.2cm
	\PfadDicke{0.5pt}
	\Pfad(0,9),111111111111111111\endPfad
	\Pfad(0,6),111111111111111111\endPfad
	\Pfad(0,3),111111111111111111\endPfad
	\Pfad(0,0),111111111111111111\endPfad
	\Pfad(18,0),222222222\endPfad
	\Pfad(15,0),222222222\endPfad
	\Pfad(12,0),222222222\endPfad
	\Pfad(9,0),222222222\endPfad
	\Pfad(6,0),222222222\endPfad
	\Pfad(3,0),222222222\endPfad
	\Pfad(0,0),222222222\endPfad
	\Label\ro{\text {\small$X$}}(1,1)
	\Label\ro{\text {\small$X$}}(4,7)
	\Label\ro{\text {\small$X$}}(7,4)
	\Label\ro{\text {\small$X$}}(10,1)
	\Label\ro{\text {\small$X$}}(13,4)
	\Label\ro{\text {\small$X$}}(16,7)
	\hskip4cm
	$$
\caption{The filling corresponds to the word $132123$.}
\label{fig:word}
\end{figure}

For a word $w$, we represent it by a filling of a rectangle.
In this filling, we label each of the four corners of each square with a partition 
and possibly label the horizontal edge of the square with a positive integer $r$,
where $r$ records the row of the corner at which the Hecke insertion terminates.
We start by labeling all corners along the bottom row and left side of the diagram 
with the partition $\emptyset$, as represented in Fig.\ref{fig:4}, and then use the following {\it forward  local rules} 
to label the other corners of the squares and some of the horizontal edges of the squares.

\begin{figure}[h]
	$$
	\begin{tikzpicture}[scale=1]
		\node (00) at (0,0) {$\emptyset$};
		\node (10) at (1,0) {$\emptyset$};
		\node (20) at (2,0) {$\emptyset$};
		\node (30) at (3,0) {$\emptyset$};
		\node (40) at (4,0) {$\emptyset$};
		\node (50) at (5,0) {$\emptyset$};
		\node (01) at (0,1) {$\emptyset$};
		\node (.51.5) at (0.5,0.5) {$X$};
		\node (11) at (1,1) {};
		\node (21) at (2,1) {};
		\node (31) at (3,1) {};
		\node (41) at (4,1) {};
		\node (51) at (5,1) {};
		\node (02) at (0,2) {$\emptyset$};
		\node (12) at (1,2) {};
		\node (22) at (2,2) {};
		\node (32) at (3,2) {};
		\node (1.50.5) at (1.5,2.5) {$X$};
		\node (42) at (4,2) {};
		\node (52) at (5,2) {};
		\node (03) at (0,3) {$\emptyset$};
		\node (13) at (1,3) {};
		\node (23) at (2,3) {};
		\node (33) at (3,3) {};
		\node (43) at (4,3) {};
		\node (53) at (5,3) {};
		\node (2.52.5) at (2.5,1.5) {$X$};
		\node (3.52.5) at (3.5,0.5) {$X$};
		\node (4.50.5) at (4.5,1.5) {$X$};
		\node (5,51.5) at (5.5,2.5) {$X$};
		\node (60) at (6,0) {$\emptyset$};
		\node (61) at (6,1) {};
		\node (62) at (6,2) {};
		\node (63) at (6,3) {};

		\draw (00)--(10)--(20)--(30)--(40)--(50)--(60)
		(01)--(11)--(21)--(31)--(41)--(51)--(61)
		(02)--(12)--(22)--(32)--(42)--(52)--(62)
		(03)--(13)--(23)--(33)--(43)--(53)--(63)
		
		(00)--(01)--(02)--(03)
		(10)--(11)--(12)--(13)
		(20)--(21)--(22)--(23)
		(30)--(31)--(32)--(33)
		(40)--(41)--(42)--(43)
		(50)--(51)--(52)--(53)
		(60)--(61)--(62)--(63);
	\end{tikzpicture}
	$$
	\caption{Initial growth diagram of word $w=132123$.}
	\label{fig:4}
\end{figure}
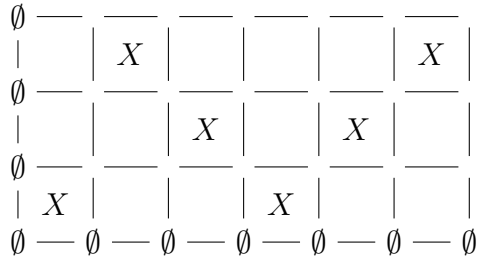

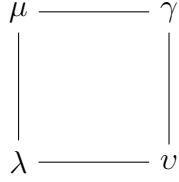
\begin{figure}[h]
	$$
	\begin{tikzpicture}[scale=2]
		\node (00) at (0,0) {$\lambda$};
		\node (10) at (1,0) {$\upsilon$};
		\node (01) at (0,1) {$\mu$};
		\node (11) at (1,1) {$\gamma$};
		
		\draw (00)--(01)
		(00)--(10)
		(11)--(10)
		(11)--(01);
	\end{tikzpicture}
	$$
	\caption{A square of a growth diagram.}
	\label{fig:5}
\end{figure}

In Fig.\ref{fig:5}, suppose we have the labels $\lambda$, $\upsilon$, $\mu$ of a square, we can construct $\gamma$ and possibly the specific positive integer for the horizontal edge of the square as follows:

Case 1. If the square contains an X:
\begin{itemize}
	\item[(F1)]If $\mu_{1}=\upsilon_{1}$, then $\gamma/\mu$ consists of one box in row $1$.
	\item[(F2)]If $\mu_{1}\neq\upsilon_{1}$, then $\gamma=\mu$ and the edge between them is labeled by $1$.
\end{itemize}

Case 2. If the square does not contain an X and either $\mu=\lambda$ or $\upsilon=\lambda$ with no label between $\lambda$ and $\upsilon$:
\begin{itemize}
	\item[(F3)] If $\mu=\lambda$, then $\gamma=\upsilon$. If $\upsilon=\lambda$, then $\gamma=\mu$.
\end{itemize}

Case 3. If the square does not contain an X, does not satisfy the condition of Case 2, and if $\upsilon\nsubseteq\mu$:
\begin{itemize}
	\item [(F4)] For Case 3, $\gamma=\upsilon\cup\mu$.
\end{itemize}

Case 4. If the square does not contain an X, does not satisfy the condition of Case 2, and if $\upsilon\subseteq\mu$:
\begin{itemize}
	\item[(F5)] If $\upsilon/\lambda$ is one box in row $i$ and $\mu/\upsilon$ has no boxes in row $i+1$, then $\gamma/\mu$ is one box in row $i+1$.
	\item[(F6)] If $\upsilon/\lambda$ is one box in row $i$ and $\mu/\upsilon$ has a box in row $i+1$, then $\gamma=\mu$ and the edge between them is labeled $i+1$.
	\item[(F7)] If $\upsilon=\lambda$, the edge between them is labeled $i$, and there are no boxes of $\mu/\upsilon$ immediately to the right or immediately below the terminate box of $\upsilon$ in row $i$, then $\gamma=\mu$ with the edge between them labeled $i$.
	\item[(F8)] If $\upsilon=\lambda$, the edge between them is labeled $i$, and there is a box of $\mu/\upsilon$ directly below the terminate box of $\upsilon$ in row $i$, then $\gamma=\mu$ with the edge between them labeled $i+1$.
	\item[(F9)] If $\upsilon=\lambda$, the edge between them is labeled $i$, and there is a box of $\mu/\upsilon$ immediately to the right of the terminate box of $\upsilon$ in row $i$ but no box of $\mu/\upsilon$ in row $i+1$, then $\gamma/\mu$ is one box in row $i+1$.
	\item[(F10)] If $\upsilon=\lambda$, the edge between them is labeled $i$, and there is a box of $\mu/\upsilon$ immediately to the right of this terminate box of $\upsilon$ in row $i$ and a box of $\mu/\upsilon$ in row $i+1$, then $\gamma=\mu$ with the edge between them labeled $i+1$. 
\end{itemize}

We call the resulting diagram the {\it Hecke growth diagram}. For example, in Fig.\ref{fig:hgd}, we have the Hecke growth diagram of the word $w=132123$ on the left.

\begin{figure}
		$$
	\begin{tikzpicture}[scale=1.2]
		\node (00) at (0,0) {$\emptyset$};
		\node (10) at (1,0) {$\emptyset$};
		\node (20) at (2,0) {$\emptyset$};
		\node (30) at (3,0) {$\emptyset$};
		\node (40) at (4,0) {$\emptyset$};
		\node (50) at (5,0) {$\emptyset$};
		\node (01) at (0,1) {$\emptyset$};
		\node (.51.5) at (0.5,0.5) {$X$};
		\node (11) at (1,1) {$1$};
		\node (21) at (2,1) {$1$};
		\node (31) at (3,1) {$1$};
		\node (41) at (4,1) {$1$};
		\node (51) at (5,1) {$1$};
		\node (02) at (0,2) {$\emptyset$};
		\node (12) at (1,2) {$1$};
		\node (22) at (2,2) {$1$};
		\node (32) at (3,2) {$2$};
		\node (1.50.5) at (1.5,2.5) {$X$};
		\node (42) at (4,2) {$21$};
		\node (52) at (5,2) {$21$};
		\node (03) at (0,3) {$\emptyset$};
		\node (13) at (1,3) {$1$};
		\node (23) at (2,3) {$2$};
		\node (33) at (3,3) {$21$};
		\node (43) at (4,3) {$211$};
		\node (53) at (5,3) {$211$};
		\node (2.52.5) at (2.5,1.5) {$X$};
		\node (3.52.5) at (3.5,0.5) {$X$};
		\node (4.50.5) at (4.5,1.5) {$X$};
		\node (5,51.5) at (5.5,2.5) {$X$};
		\node (60) at (6,0) {$\emptyset$};
		\node (61) at (6,1) {$1$};
		\node (62) at (6,2) {$21$};
		\node (63) at (6,3) {$311$};
		\node at (3.5,1.15) {\tiny $1$};
		\node at (4.5,2.15) {\tiny $1$};
		\node at (4.5,3.15) {\tiny $1$};

		\draw (00)--(10)--(20)--(30)--(40)--(50)--(60)
		(01)--(11)--(21)--(31)--(41)--(51)--(61)
		(02)--(12)--(22)--(32)--(42)--(52)--(62)
		(03)--(13)--(23)--(33)--(43)--(53)--(63)
		
		(00)--(01)--(02)--(03)
		(10)--(11)--(12)--(13)
		(20)--(21)--(22)--(23)
		(30)--(31)--(32)--(33)
		(40)--(41)--(42)--(43)
		(50)--(51)--(52)--(53)
		(60)--(61)--(62)--(63);

	\end{tikzpicture}
	\hbox{\hskip1.5cm}
	\begin{tikzpicture}[scale=1.2]
		\node (00) at (0,0) {$\emptyset$};
		\node (10) at (1,0) {$\emptyset$};
		\node (20) at (2,0) {$\emptyset$};
		\node (30) at (3,0) {$\emptyset$};
		\node (40) at (4,0) {$\emptyset$};
		\node (50) at (5,0) {$\emptyset$};
		\node (01) at (0,1) {$\emptyset$};
		\node (11) at (1,1) {$1$};
		\node (21) at (2,1) {$1$};
		\node (31) at (3,1) {$1$};
		\node (41) at (4,1) {$1$};
		\node (51) at (5,1) {$1$};
		\node (02) at (0,2) {$\emptyset$};
		\node (12) at (1,2) {$1$};
		\node (22) at (2,2) {$1$};
		\node (32) at (3,2) {$2$};
		\node (42) at (4,2) {$21$};
		\node (52) at (5,2) {$21$};
		\node (03) at (0,3) {$\emptyset$};
		\node (13) at (1,3) {$1$};
		\node (23) at (2,3) {$2$};
		\node (33) at (3,3) {$21$};
		\node (43) at (4,3) {$211$};
		\node (53) at (5,3) {$211$};
		\node (60) at (6,0) {$\emptyset$};
		\node (61) at (6,1) {$1$};
		\node (62) at (6,2) {$21$};
		\node (63) at (6,3) {$311$};
		\node at (3.5,1.15) {\tiny $1$};
		\node at (4.5,2.15) {\tiny $1$};
		\node at (4.5,3.15) {\tiny $1$};
		\node at (0.5,0.5) {$F1$};
		\node at (0.5,1.5) {$F3$};
		\node at (0.5,2.5) {$F3$};
		\node at (1.5,0.5) {$F3$};
		\node at (1.5,1.5) {$F3$};
		\node at (1.5,2.5) {$F1$};
		\node at (2.5,0.5) {$F3$};
		\node at (2.5,1.5) {$F1$};
		\node at (2.5,2.5) {$F5$};
		\node at (3.5,0.5) {$F2$};
		\node at (3.5,1.5) {$F9$};
		\node at (3.5,2.5) {$F5$};
		\node at (4.5,0.5) {$F3$};
		\node at (4.5,1.5) {$F2$};
		\node at (4.5,2.5) {$F7$};
		\node at (5.5,0.5) {$F3$};
		\node at (5.5,1.5) {$F3$};
		\node at (5.5,2.5) {$F1$};

		\draw (00)--(10)--(20)--(30)--(40)--(50)--(60)
		(01)--(11)--(21)--(31)--(41)--(51)--(61)
		(02)--(12)--(22)--(32)--(42)--(52)--(62)
		(03)--(13)--(23)--(33)--(43)--(53)--(63)
		
		(00)--(01)--(02)--(03)
		(10)--(11)--(12)--(13)
		(20)--(21)--(22)--(23)
		(30)--(31)--(32)--(33)
		(40)--(41)--(42)--(43)
		(50)--(51)--(52)--(53)
		(60)--(61)--(62)--(63);

	\end{tikzpicture}
	$$
	\caption{The Hecke growth diagram of the word 132123 on the left and the local rules applied at each step to build it on the right.}
	\label{fig:hgd}
\end{figure}
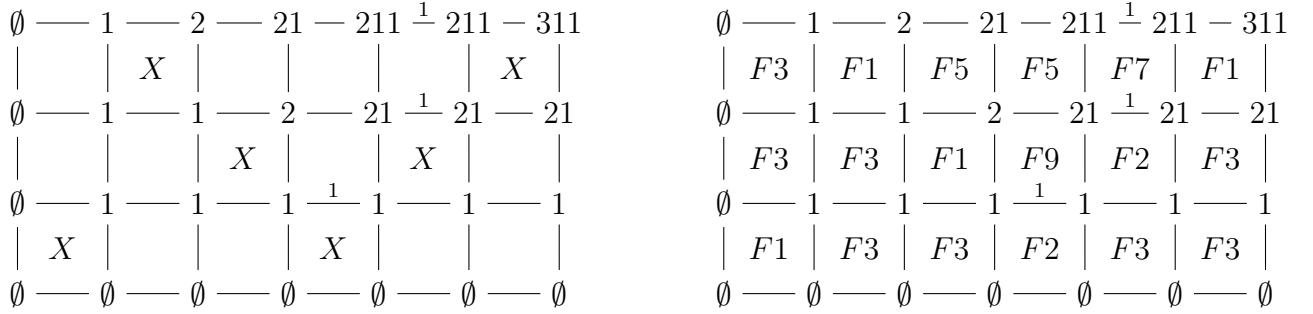

Let $\emptyset=\mu_{0}\subseteq\mu_{1}\subseteq...\subseteq\mu_{n}$ be the sequence of partitions along the top border of the Hecke growth diagram of a word $w$, and let $\emptyset=\upsilon_{0}\subseteq\upsilon_{1}\subseteq...\subseteq\upsilon_{m}$ be the sequence of partitions along the right border of the Hecke growth diagram. The sequence along the right border corresponds to a straight increasing tableau $P(w)$, and the sequence along the top border corresponds to a set-valued tableau $Q(w)$. For $P(w)$, we put $i$'s into the squares of $\upsilon_{i}/\upsilon_{i-1}$. For $Q(w)$, if $\mu_{i-1}\neq\mu_{i}$, then we put $i$ into the square of $\mu_{i}/\mu_{i-1}$. Otherwise, $\mu_{i-1}=\mu_{i}$ and the edge between them is label $r$, then we put $i$ into the box at the end of row $r$ of $\mu_{i-1}$. For example, from the Hecke growth diagram on the left in Fig.\ref{fig:hgd}, we obtain $P(132123)$ and $Q(132123)$ as shown in Fig.\ref{fig:7}.

\begin{figure}
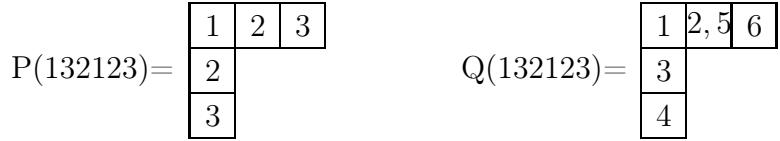

	$$
	\Einheit.2cm
	\PfadDicke{.5pt}
	\Pfad(0,0),222222222\endPfad
	\Pfad(3,0),222222222\endPfad
	\Pfad(6,6),222\endPfad
	\Pfad(9,6),222\endPfad
	\Pfad(0,0),111\endPfad
	\Pfad(0,3),111\endPfad
	\Pfad(0,6),111111111\endPfad
	\Pfad(0,9),111111111\endPfad
	\Label\ro{1}(1,7)
	\Label\ro{2}(1,4)
	\Label\ro{2}(4,7)
	\Label\ro{3}(1,1)
	\Label\ro{3}(7,7)
	\Label\ro{$P(132123)=$}(-7,4)
	\hbox{\hskip6cm}
	\Einheit.2cm
	\PfadDicke{.5pt}
	\Pfad(0,0),222222222\endPfad
	\Pfad(3,0),222222222\endPfad
	\Pfad(6,6),222\endPfad
	\Pfad(9,6),222\endPfad
	\Pfad(0,0),111\endPfad
	\Pfad(0,3),111\endPfad
	\Pfad(0,6),111111111\endPfad
	\Pfad(0,9),111111111\endPfad
	\Label\ro{1}(1,7)
	\Label\ro{3}(1,4)
	\Label\ro{2,5}(4,7)
	\Label\ro{4}(1,1)
	\Label\ro{6}(7,7)
	\Label\ro{$Q(132123)=$}(-7,4)
	\hskip1.8cm
	$$
	\caption{$P(132123)$ and $Q(132123)$ obtained from the Hecke growth diagram in Fig.\ref{fig:hgd}.}
	\label{fig:7}
\end{figure}

It is well known that there is a correspondence between a permutation $\pi$ and a pair of standard Young tableaux $(P, Q)$ via the RSK algorithm. This correspondence can also be obtained using Fomin's growth diagram \cite{Fomin86,Fomin94,Fomin95}. There is an analogous result for the Hecke growth diagram of words, as shown in the following theorem.

\begin{theorem}
	\cite[Theorem 4.16]{PP}
	For any word $w$, the increasing tableau $P(w)$ and set-valued tableau $Q(w)$ obtained from the sequence of partitions along the right border of the Hecke growth diagram for $w$
	and along the upper border of the Hecke growth diagram for $w$, respectively, are the Hecke insertion tableau and the Hecke recording tableau for $w$.
	\label{PQ-w}
\end{theorem}

\subsection{Jeu de taquin for increasing tableaux}
$\newline$

In this section, we describe the jeu de taquin for increasing tableaux which was introduced by Thomas and Yong in \cite{TY09}.

For $T\in$ INC($\rho/\lambda$), an {\it inner corner} is a maximally southeast box of $\lambda$. Let $C$ be a set of some inner corners of $\lambda$, and mark these boxes in $C$ with $\bullet$. The {\it jeu de taquin} for $T$ and $C$ is described as follows. We begin with the tableau $T_{0}$ which is composed of $T\cup C$. If $1$ is directly below or to the right of $\bullet$, then we swap $\bullet$ and $1$ in $T_{0}$; otherwise, we leave $T_{0}$ unchanged. We denote the resulting tableau of $T_{0}$ by $T_{1}$. Next we consider the tableau $T_{1}$. If $2$ is directly below or to the right of $\bullet$, then we swap $\bullet$ and $2$ in $T_{1}$; otherwise, we leave $T_{1}$ unchanged. We denote the resulting tableau of $T_{1}$ by $T_{2}$. We repeat the above processes until the $\bullet$ reach the inner corners of $\rho$. We denote the final tableau as $jdt_{C}(T)$. It is easy to see that $jdt_{C}(T)$ is also an increasing tableau.
Let $\rho=(3,3,2)$ and $\lambda=(2,1)$. Let $T\in$ INC($\rho/\lambda$) be given as in Fig.\ref{F jdt}, and $C$ indicated with $\bullet$. See Fig.\ref{F jdt}e for $jdt_{C}(T)$.

\begin{figure}
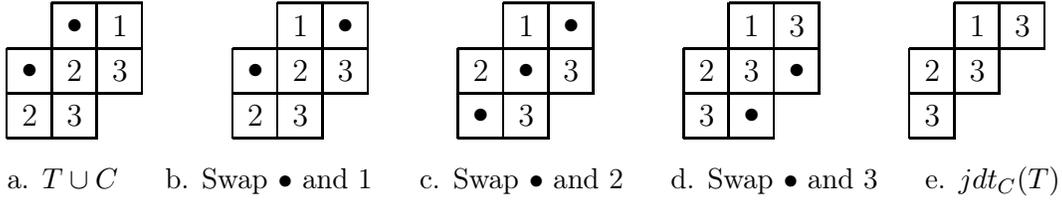

	$$
	\Einheit.2cm
	\PfadDicke{.5pt}
	\Pfad(0,0),222222\endPfad
	\Pfad(3,0),222222222\endPfad
	\Pfad(6,0),222222222\endPfad
	\Pfad(9,3),222222\endPfad
	\Pfad(0,0),111111\endPfad
	\Pfad(0,3),111111111\endPfad
	\Pfad(0,6),111111111\endPfad
	\Pfad(3,9),111111\endPfad
	\Label\ro{2}(1,1)
	\Label\ro{3}(4,1)
	\Label\ro{2}(4,4)
	\Label\ro{3}(7,4)
	\Label\ro{1}(7,7)
	\Label\ro{\bullet}(1,4)
	\Label\ro{\bullet}(4,7)
	\hbox{\hskip3cm}
	\Einheit.2cm
	\PfadDicke{.5pt}
	\Pfad(0,0),222222\endPfad
	\Pfad(3,0),222222222\endPfad
	\Pfad(6,0),222222222\endPfad
	\Pfad(9,3),222222\endPfad
	\Pfad(0,0),111111\endPfad
	\Pfad(0,3),111111111\endPfad
	\Pfad(0,6),111111111\endPfad
	\Pfad(3,9),111111\endPfad
	\Label\ro{2}(1,1)
	\Label\ro{3}(4,1)
	\Label\ro{2}(4,4)
	\Label\ro{3}(7,4)
	\Label\ro{\bullet}(7,7)
	\Label\ro{\bullet}(1,4)
	\Label\ro{1}(4,7)
	\hbox{\hskip3cm}
	\Einheit.2cm
	\PfadDicke{.5pt}
	\Pfad(0,0),222222\endPfad
	\Pfad(3,0),222222222\endPfad
	\Pfad(6,0),222222222\endPfad
	\Pfad(9,3),222222\endPfad
	\Pfad(0,0),111111\endPfad
	\Pfad(0,3),111111111\endPfad
	\Pfad(0,6),111111111\endPfad
	\Pfad(3,9),111111\endPfad
	\Label\ro{\bullet}(1,1)
	\Label\ro{3}(4,1)
	\Label\ro{\bullet}(4,4)
	\Label\ro{3}(7,4)
	\Label\ro{\bullet}(7,7)
	\Label\ro{2}(1,4)
	\Label\ro{1}(4,7)
	\hbox{\hskip3cm}
	\Einheit.2cm
	\PfadDicke{.5pt}
	\Pfad(0,0),222222\endPfad
	\Pfad(3,0),222222222\endPfad
	\Pfad(6,0),222222222\endPfad
	\Pfad(9,3),222222\endPfad
	\Pfad(0,0),111111\endPfad
	\Pfad(0,3),111111111\endPfad
	\Pfad(0,6),111111111\endPfad
	\Pfad(3,9),111111\endPfad
	\Label\ro{3}(1,1)
	\Label\ro{\bullet}(4,1)
	\Label\ro{3}(4,4)
	\Label\ro{\bullet}(7,4)
	\Label\ro{3}(7,7)
	\Label\ro{2}(1,4)
	\Label\ro{1}(4,7)
	\hbox{\hskip3cm}
	\Einheit.2cm
	\PfadDicke{.5pt}
	\Pfad(0,0),222222\endPfad
	\Pfad(3,0),222222222\endPfad
	\Pfad(6,3),222222\endPfad
	\Pfad(9,6),222\endPfad
	\Pfad(0,0),111\endPfad
	\Pfad(0,3),111111\endPfad
	\Pfad(0,6),111111111\endPfad
	\Pfad(3,9),111111\endPfad
	\Label\ro{3}(1,1)
	\Label\ro{3}(4,4)
	\Label\ro{3}(7,7)
	\Label\ro{2}(1,4)
	\Label\ro{1}(4,7)
	\hskip2cm
	$$
	\centerline{\small a. $T\cup C$
		\hskip0.5cm
		b. Swap $\bullet$ and 1
		\hskip0.5cm
		c. Swap $\bullet$ and 2
		\hskip0.5cm
		d. Swap $\bullet$ and 3
		\hskip0.5cm
		e. $jdt_{C}(T)$}
	\caption{An example of jeu de taquin for increasing tableaux.}
	\label{F jdt}
\end{figure}

The {\it reverse jeu de taquin} for $T'\in$ INC($\rho/\lambda$) is performed in a similar manner. For $T'\in$ INC($\rho/\lambda$), an {\it outer corner} is a box that shares an edge with the southeast side of $\rho/\lambda$ and can be added such that the resulting shape is still a skew shape. Let $C'$ be a set of some outer corners, and mark these boxes in $C'$ with $\bullet$. We begin with the tableau $T'_{n}$ which is made of $T'\cup C'$. If $n$ (the maximum entry of $T'$) is directly above or to the left of $\bullet$, then we swap $\bullet$ and $n$ in $T'_{n}$; otherwise, we leave $T'_{n}$ unchanged. We denoted the resulting tableau of $T'_{n}$ by $T'_{n-1}$. Next, we consider the tableau $T'_{n-1}$. If $n-1$ is directly above or to the left of $\bullet$, then we swap $\bullet$ and $n-1$ in $T'_{n-1}$; otherwise, we leave $T'_{n-1}$ unchanged. We denote the resulting tableau of $T'_{n-1}$ by $T'_{n-2}$. We repeat the above processes until the $\bullet$ reach the inner corners of $\lambda$. We denote the final tableau by $rjdt_{C'}(T')$. See Fig.\ref{R jdt} for an example.

\begin{figure}
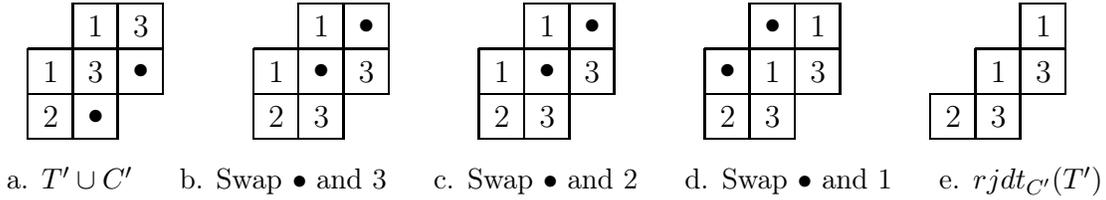

	$$
	\Einheit.2cm
	\PfadDicke{.5pt}
	\Pfad(0,0),222222\endPfad
	\Pfad(3,0),222222222\endPfad
	\Pfad(6,0),222222222\endPfad
	\Pfad(9,3),222222\endPfad
	\Pfad(0,0),111111\endPfad
	\Pfad(0,3),111111111\endPfad
	\Pfad(0,6),111111111\endPfad
	\Pfad(3,9),111111\endPfad
	\Label\ro{2}(1,1)
	\Label\ro{\bullet}(4,1)
	\Label\ro{3}(4,4)
	\Label\ro{\bullet}(7,4)
	\Label\ro{3}(7,7)
	\Label\ro{1}(1,4)
	\Label\ro{1}(4,7)
	\hbox{\hskip3cm}
	\Einheit.2cm
	\PfadDicke{.5pt}
	\Pfad(0,0),222222\endPfad
	\Pfad(3,0),222222222\endPfad
	\Pfad(6,0),222222222\endPfad
	\Pfad(9,3),222222\endPfad
	\Pfad(0,0),111111\endPfad
	\Pfad(0,3),111111111\endPfad
	\Pfad(0,6),111111111\endPfad
	\Pfad(3,9),111111\endPfad
	\Label\ro{2}(1,1)
	\Label\ro{3}(4,1)
	\Label\ro{\bullet}(4,4)
	\Label\ro{3}(7,4)
	\Label\ro{\bullet}(7,7)
	\Label\ro{1}(1,4)
	\Label\ro{1}(4,7)
	\hbox{\hskip3cm}
	\Einheit.2cm
	\PfadDicke{.5pt}
	\Pfad(0,0),222222\endPfad
	\Pfad(3,0),222222222\endPfad
	\Pfad(6,0),222222222\endPfad
	\Pfad(9,3),222222\endPfad
	\Pfad(0,0),111111\endPfad
	\Pfad(0,3),111111111\endPfad
	\Pfad(0,6),111111111\endPfad
	\Pfad(3,9),111111\endPfad
	\Label\ro{2}(1,1)
	\Label\ro{3}(4,1)
	\Label\ro{\bullet}(4,4)
	\Label\ro{3}(7,4)
	\Label\ro{\bullet}(7,7)
	\Label\ro{1}(1,4)
	\Label\ro{1}(4,7)
	\hbox{\hskip3cm}
	\Einheit.2cm
	\PfadDicke{.5pt}
	\Pfad(0,0),222222\endPfad
	\Pfad(3,0),222222222\endPfad
	\Pfad(6,0),222222222\endPfad
	\Pfad(9,3),222222\endPfad
	\Pfad(0,0),111111\endPfad
	\Pfad(0,3),111111111\endPfad
	\Pfad(0,6),111111111\endPfad
	\Pfad(3,9),111111\endPfad
	\Label\ro{2}(1,1)
	\Label\ro{3}(4,1)
	\Label\ro{1}(4,4)
	\Label\ro{3}(7,4)
	\Label\ro{1}(7,7)
	\Label\ro{\bullet}(1,4)
	\Label\ro{\bullet}(4,7)
	\hbox{\hskip3cm}
	\Einheit.2cm
	\PfadDicke{.5pt}
	\Pfad(0,0),222\endPfad
	\Pfad(3,0),222222\endPfad
	\Pfad(6,0),222222222\endPfad
	\Pfad(9,3),222222\endPfad
	\Pfad(0,0),111111\endPfad
	\Pfad(0,3),111111111\endPfad
	\Pfad(3,6),111111\endPfad
	\Pfad(6,9),111\endPfad
	\Label\ro{2}(1,1)
	\Label\ro{3}(4,1)
	\Label\ro{1}(4,4)
	\Label\ro{3}(7,4)
	\Label\ro{1}(7,7)
	\hskip2cm
	$$
	\centerline{\small a. $T'\cup C'$
		\hskip0.5cm
		b. Swap $\bullet$ and 3
		\hskip0.5cm
		c. Swap $\bullet$ and 2
		\hskip0.5cm
		d. Swap $\bullet$ and 1
		\hskip0.5cm
		e. $rjdt_{C'}(T')$}
	\caption{An example of reverse jeu de taquin for increasing tableaux.}
	\label{R jdt}
\end{figure}

Two increasing skew tableaux $T$ and $T'$ are {\it K-jeu de taquin equivalent}\, if $T$ can be obtained by applying a sequence of jeu de taquin or reverse jeu de taquin operations to $T'$ \cite{BuchSamuel}.

\subsection{$K$-Knuth equivalence on words and increasing tableaux}
$\newline$

Two permutations are {\it Knuth equivalent} if one can be obtained from the other via a finite series of applications of the Knuth relations: (1) for $x<y<z$, $xzy\sim zxy$; (2) for $x<y<z$, $yxz\sim yzx$. Two permutations are Knuth equivalent if and only if they have the same insertion tableau by the RSK algorithm. 

Knuth equivalence for permutations has been extended to the context of words, as defined by Buch and Samuel in \cite{BuchSamuel}. Two words are said to be {\it K-Knuth equivalent} if one can be transformed into the other through a finite series of applications of the following $K$-Knuth relations:

\begin{align*}
	xzy &\equiv zxy,\qquad (x < y < z) \\
	yxz &\equiv yzx,\qquad (x < y < z) \\
	x &\equiv xx, \\
	xyx &\equiv yxy.
\end{align*}

Let $T$ be an increasing tableau, and we denote row($T$) as the reading word of $T$, obtained by reading the entries of $T$ from left to right along each row, starting from the bottom row to the top. For example, the reading word of $P(w)$ in Fig.\ref{(P,Q) for word} is $54341235$. For two increasing tableaux $T$ and $T'$, if row($T$) $\equiv$ row($T'$), we say $T$ and $T'$ are {\it K-Knuth equivalent}, denoted by $T\equiv T'$. Here are some results that we will be used in Section 4.

\begin{lemma}\cite[Lemma 5.5]{BuchSamuel}
	Let $[a,b]$ be an integer interval. Let $w_1$ and $w_2$ be $K$-Knuth equivalent words. 
	For $i=1,2$, let $w_i|_{[a,b]}$ be the word obtained from $w_i$ by deleting all integers not 
	contained in the interval $[a,b]$. Then $w_1|_{[a,b]}$ and $w_2|_{[a,b]}$ are
	$K$-Knuth equivalent words.
	\label{pro1}
\end{lemma}

\begin{theorem}\cite[Theorem 6.2]{BuchSamuel}
	Let $T$ and $T^\prime$ be increasing tableaux. Then $T$ and $T^\prime$ are $K$-Knuth equivalent
	if and only if $T$ and $T^\prime$ are $K$-jeu de taquin equivalent.
	\label{pro2}
\end{theorem}

\begin{lemma}\cite[Lemma 58]{GMPPRST}
	If $w$ be a word and $P(w)$ be the Hecke insertion tableau of $w$, then $w$ and row$(P(w))$ are $K$-Knuth equivalent.
	\label{pro3}
\end{lemma}

\vspace{0.5cm}

\section{Backward local rules}

As seen in Section 2, the forward local rules for the growth diagram of Hecke insertion were designed by Patrias and Pylyavskyy in \cite{PP}. However, they did not provide the backward local rules, that is, given $\gamma,\mu,\lambda$ and the edge label between $\mu$ and $\gamma$, one hope to reconstruct $\lambda$, the edge label between $\lambda$ and $\upsilon$, and the filling of the square. In this section, we formulate the {\it backward local rules} which will be used in Section 4 to prove our main theorem. In describing the conclusions for the following 16 cases, we do not explicitly state in the conclusions when there is no $X$ in the square and there is no label between $\lambda$ and $\nu$, for the sake of brevity.

\vspace{0.5cm}

\textbf {Case 1. There is no label between $\mu$ and $\gamma$.}
\begin{itemize}
	\item [(B1)] If $\mu=\gamma$, then $\lambda=\upsilon$.
	\item [(B2)] If $\mu\neq\gamma$ and $\upsilon=\gamma$, then $\lambda=\mu$.
	\item [(B3)] If $\gamma/\mu$ is one box in the first row and $\gamma/\upsilon$ has a box in the first row, then $\lambda=\upsilon$ and the square contains an $X$.
	\item [(B4)] If $\gamma/\mu$ is one box in the first row and $\gamma_{1}=\upsilon_{1}$, then $\upsilon/\lambda$ is one box in the first row.
	\item [(B5)] If $\gamma/\mu$ is one box in row $i$ ($i\geq2$) and $\upsilon_{i}=\gamma_{i}$, then $\upsilon/\lambda$ is one box in row $i$.
	\item [(B6)] If $\gamma/\mu$ is one box in row $i$ ($i\geq2$), $\upsilon_{i}=\gamma_{i}-1$ and $\upsilon_{i-1}=\gamma_{i-1}$, then $\upsilon/\lambda$ is one box in row $i-1$.
	\item [(B7)] If $\gamma/\mu$ is one box in row $i$ ($i\geq2$), $\upsilon_{i}=\gamma_{i}-1$, and $\upsilon_{i-1}=\gamma_{i-1}-1$, then $\lambda=\upsilon$ and the edge between them is labeled $i-1$.
\end{itemize}

\textbf {Case 2. There is a label $r$ between $\mu$ and $\gamma$.}
\begin{itemize}
	\item [(B8)] If $r=1$ and $\upsilon=\gamma$, then $\lambda=\mu$ and the edge between $\lambda$ and $\upsilon$ is labeled $1$.
	\item [(B9)] If $r=1$ and $\upsilon_{1}\neq\gamma_{1}$, then $\lambda=\upsilon$ and the square contains an $X$.
	\item [(B10)] If $r=1$, $\upsilon\neq\gamma$ and $\upsilon_{1}=\gamma_{1}$, then $\lambda=\upsilon$ and the edge between them is labeled $1$.
	\item [(B11)] If $r\geq2$ and $\upsilon=\gamma$, then $\lambda=\mu$ and the edge between $\lambda$ and $\upsilon$ is labeled $r$.
	\item [(B12)] If $r\geq2$, $\gamma_{r}=\gamma_{r-1}$ and $\gamma_{r}=\upsilon_{r}$, then $\lambda=\upsilon$ and the edge between them is labeled $r$.
	\item [(B13)] If $r\geq2$, $\gamma_{r}=\gamma_{r-1}$ and $\gamma_{r}=\upsilon_{r}+1$, then $\lambda=\upsilon$ and the edge between them is labeled $r-1$.
	\item [(B14)] If $r\geq2$, $\gamma_{r}<\gamma_{r-1}$ and $\gamma_{r}=\upsilon_{r}$, then $\lambda=\upsilon$ and the edge between them is labeled $r$.
	\item [(B15)] If $r\geq2$, $\gamma_{r}<\gamma_{r-1}$, $\gamma_{r}=\upsilon_{r}+1$ and $\gamma_{r-1}=\upsilon_{r-1}$, then $\upsilon/\lambda$ is one box in row $r-1$.
	\item [(B16)] If $r\geq2$, $\gamma_{r}<\gamma_{r-1}$, $\gamma_{r}=\upsilon_{r}+1$ and $\gamma_{r-1}=\upsilon_{r-1}+1$, then $\lambda=\upsilon$ and the edge between them is labeled $r-1$.
\end{itemize}

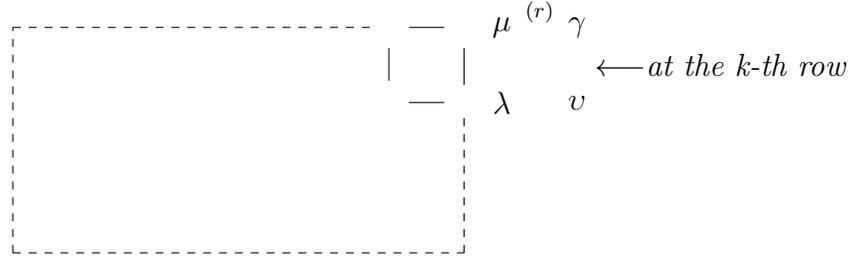
\begin{figure}
	$$
	\begin{tikzpicture}[scale=1]
		\hskip1.5cm
		\node (52) at (5,2) {$\lambda$};
		\node (62) at (6,2) {$\upsilon$};
		\node (53) at (5,3) {$\mu$};
		\node (63) at (6,3) {$\gamma$};
		\node at (7.9,2.5) {$\longleftarrow${\it at the k-th row}};
		\node at (5.5,3.2) {\tiny $(r)$};
		
		\draw (52)--(62)
		(52)--(53)
		(63)--(62)
		(63)--(53);
		\draw [dashed] (0,0)--(0,3)
		(0,3)--(4.8,3)
		(0,0)--(6,0)
		(6,0)--(6,1.8);
	\end{tikzpicture}
	$$
	\caption{A square of a rectangular growth diagram.}
	\label{fig:10}
\end{figure}

Before proving these backward local rules, we need the following three necessary facts.

\begin{claim}
	\cite[Claim 4.17]{PP} In Hecke growth diagram, for any partition $\mu$ and its right neighbor partition $\gamma$, the two partitions differ by at most one box, i.e. $|\gamma/\mu|\leq1$.
    \label{cl:3.2}
\end{claim}

\begin{claim}
	\cite[Claim 4.18]{PP} In Hecke growth diagram, for any partition $\upsilon$ and its top neighbor partition $\gamma$, $\gamma/\upsilon$ is a rook strip, that is, no two boxes in $\gamma/\upsilon$ are in the same row or column.
	\label{cl:3.3}
\end{claim}

\begin{theorem}
	\cite[Proposition 3.2]{CGP} Let $w=w_1w_2\cdots w_n$ be a word of positive integers, and $k$ be the maximal element appearing in $w$.
	Let $w'=a_1a_2\cdots a_m$ be the word obtained from $w$ by deleting the elements equal to $k$.
	Assume that $T$ is the Hecke insertion tableau of $w$ and $T'$ is the Hecke insertion tableau of $w'$.
	Then $T'$ is obtained from $T$ by deleting the squares occupied with $k$.
	\label{th:3.4}
\end{theorem}

Our proof is based on the Hecke insertion algorithm. For convenience, we assume the position of the square in Fig.\ref{fig:10} is in the $k$-th row of the rectangle.

\begin{proof}
	Let the Hecke insertion tableaux of $\mu$, $\upsilon$, and $\gamma$ be $P_{\mu}$, $P_{\upsilon}$, and $P_{\gamma}$, respectively. A {\it $k$-box} is a box that contains the letter $k$. 
	\\
	\item [(B1)] Since $\mu=\gamma$ and there is no label between $\mu$ and $\gamma$, this implies there is no $X$ below $\mu$ and $\gamma$. Thus, there is no $X$ below $\lambda$ and $\upsilon$. This means that $\lambda$ and $\upsilon$ have the same Hecke insertion tableau, and there is no label between them. Therefore, $\lambda=\upsilon$.
	\\
	\item [(B2)] Since $\upsilon=\gamma$, there is no $X$ to the left of $\upsilon$ and $\gamma$. Thus, there is no $X$ to the left of $\lambda$ and $\mu$. This means that $\lambda$ and $\mu$ have the same Hecke insertion tableau. Therefore, $\lambda=\mu$.
	\\
	\item [(B3)] Suppose there is a $k$-box in the first row of $P_{\mu}$. Let $P_{\gamma}=(P_{\mu}\xleftarrow{\,\mathrm{H}}l)$, where $l\leq k$. By the Hecke insertion algorithm, we get $\mu_{1}=\gamma_{1}$, which contradicts the fact that $\gamma/\mu$ is one box in the first row. Therefore, there is no $k$-box in the first row of $P_{\mu}$. Since $\gamma/\upsilon$ has a box in the first row, by Theorem \ref{th:3.4}, $P_{\gamma}$ must have a $k$-box. This implies that the algorithm $P_{\gamma}=(P_{\mu}\xleftarrow{\,\mathrm{H}}l)$ places a $k$-box at the end of the first row of $P_{\mu}$. Thus, there is an $X$ in the square and no $X$ below $\lambda$ and $\upsilon$, therefore $\lambda=\upsilon$.
	\\
	\item [(B4)] By the proof of (B3), we know that there is no $k$-box in the first row of $P_{\mu}$. Thus, by Theorem \ref{th:3.4}, we have $\mu_{1}=\lambda_{1}$. Since $\upsilon_{1}=\gamma_{1}=\mu_{1}+1=\lambda_{1}+1$, by Claim \ref{cl:3.2}, $\lambda$ is obtained by deleting one box from the first row of $\upsilon$. Thus, $\upsilon/\lambda$ is one box in the first row.
	\\
	\item [(B5)] Since $\gamma/\mu$ is one box in row $i$, there is no $k$-box in the $i$-th row of $P_{\mu}$. Therefore, $\mu_{i}=\lambda_{i}$ (cf. Theorem \ref{th:3.4}). Since $\upsilon_{i}=\gamma_{i}=\mu_{i}+1=\lambda_{i}+1$, by Claim \ref{cl:3.2}, $\upsilon/\lambda$ is one box in row $i$.
	\\
	\item [(B6)] From the proof of (B5), we know that there is no $k$-box in the $i$-th row of $P_{\mu}$. Since $\upsilon_{i}=\gamma_{i}-1$, $\gamma$ has a $k$-box in row $i$. 
	By algorithm $P_{\gamma}=(P_{\mu}\xleftarrow{\,\mathrm{H}}l)$ ($l\leq k$), we can conclude that the $k$-box in row $i$ of $P_{\gamma}$ results from inserting integer $k$ into the $i$-th row of $P_{\mu}$. Thus, the  $(i-1)$-th row of $P_{\mu}$ has a $k$-box. Therefore, $\lambda_{i-1}=\mu_{i-1}-1$ and $\upsilon_{i-1}=\gamma_{i-1}=\mu_{i-1}=\lambda_{i-1}+1$. By Claim \ref{cl:3.2}, $\upsilon/\lambda$ is one box in row $i-1$.
	\\
	\item [(B7)] From the proof of (B6), we know that there is no $k$-box in row $i$ of $P_{\mu}$, $P_{\gamma}$ has a $k$-box in row $i$, and $P_{\mu}$ has a $k$-box in row $i-1$. Since $\upsilon_{i-1}=\gamma_{i-1}-1$, there is a $k$-box in row $i-1$ of $P_{\gamma}$. Similar to the proof of (B6), the $k$-box in row $i$ of $P_{\gamma}$ results from inserting integer $k$ into the $i$-th row of $P_{\mu}$. Since $P_{\gamma}$ has a $k$-box in row $i-1$, some integer $x$ was inserted into the $(i-1)$-th row of $P_{\mu}$, which did not change this row but produced an output integer $k$. Therefore, while $x$ was inserted into the $(i-1)$-th row of $P_{\lambda}$, this row did not change and the insertion algorithm $P_{\upsilon}=(P_{\lambda}\xleftarrow{\,\mathrm{H}}l)$ terminated, with the terminate box being the rightmost box in row $i-1$ of $P_{\lambda}$. Thus, $P_{\lambda}$ and $P_{\upsilon}$ have the same shape, i.e. $\lambda=\upsilon$, and the edge between them is labeled $i-1$.
	\\
	\item [(B8)] By the proof of (B2), we know that there is no $X$ to the left of $\upsilon$ and $\gamma$ and  $\lambda=\mu$. Therefore, according to Hecke insertion algorithm, $P_{\gamma}=(P_{\mu}\xleftarrow{\,\mathrm{H}}l)\Leftrightarrow P_{\upsilon}=(P_{\lambda}\xleftarrow{\,\mathrm{H}}l)$, where $l\leq k$. The edge between $\lambda$ and $\upsilon$ is labeled $1$.
	\\
	\item [(B9)] Since $\upsilon_{1}\neq\gamma_{1}$, by Theorem \ref{th:3.4}, there is a $k$-box in the first row of $P_{\gamma}$. Since $\mu=\gamma$ and $r=1$, according to the Hecke insertion algorithm, the first row of $P_{\mu}$ is identical to the first row of $P_{\gamma}$, that is, $P_{\mu}$ has a $k$-box in the first row and $P_{\gamma}=(P_{\mu}\xleftarrow{\,\mathrm{H}}k)$. Thus, there is an $X$ in the square, and no $X$ appears below $\lambda$ and $\upsilon$. Therefore, $\lambda=\upsilon$.
	\\
	\item [(B10)] Since $\upsilon_{1}=\gamma_{1}$ and $r=1$, there is no $k$-box in the first row of $P_{\gamma}$ and $P_{\mu}$. Hence, $\lambda_{1}=\mu_{1}$, and the process of the Hecke insertion algorithm $P_{\gamma}=(P_{\mu}\xleftarrow{\,\mathrm{H}}l)$ in the first row of $P_{\mu}$ is the same as the process for $P_{\upsilon}=(P_{\lambda}\xleftarrow{\,\mathrm{H}}l)$ in the first row of $P_{\lambda}$, where $l\leq k$. Therefore, $\lambda=\upsilon$, and the edge between them is labeled $1$.
	\\
	\item [(B11)] The proof is the same as (B8).
	\\
	\item [(B12)] Since $\gamma_{r}=\upsilon_{r}$, there is no $k$-box in row $r$ of $P_{\gamma}$. Thus, the terminate box $c$ of $P_{\gamma}$ is also the terminate box of $P_{\upsilon}$, and $k$ is not replaced as an output integer during the algorithm. This implies that $P_{\gamma}=(P_{\mu}\xleftarrow{\,\mathrm{H}}l)\Leftrightarrow P_{\upsilon}=(P_{\lambda}\xleftarrow{\,\mathrm{H}}l)$, where $l\leq k$. Therefore, $\lambda=\upsilon$, and the edge between them is labeled $r$.
	\\
	\item [(B13)] Since $\gamma_{r}=\upsilon_{r}+1$, there is a $k$-box in the $r$-th row of $P_{\gamma}$, and the terminate box of $P_{\gamma}$ is the $k$-box. Since the edge between $\mu$ and $\gamma$ is labeled $r$, by Hecke insertion algorithm, the $r$-th row of $P_{\mu}$ is identical to the $r$-th row of $P_{\gamma}$, that is, there is a $k$-box in the $r$-th row of $P_{\mu}$ (e.g. Fig.\ref{fig:11}a). So, by Hecke insertion algorithm, while $P_{\gamma}=(P_{\mu}\xleftarrow{\,\mathrm{H}}l)$ terminates at $k$-box in row $r$ of $P_{\mu}$, $P_{\upsilon}=(P_{\lambda}\xleftarrow{\,\mathrm{H}}l)$ terminates at  $x$-box in row $r-1$ of $P_{\lambda}$. Thus, $\lambda=\upsilon$ and the edge between them is labeled $r-1$.
	\\
	\item [(B14)] The proof is the same as (B12).
	\\
	\item [(B15)] Since $\gamma_{r}=\upsilon_{r}+1$ and $\gamma_{r-1}=\upsilon_{r-1}$, there is a $k$-box in row $r$ of $P_{\gamma}$ and no $k$-box in row $r-1$. Because the edge between $\mu$ and $\gamma$ is labeled $r$, the $k$-box in row $r$ of $P_{\gamma}$ is the terminate box of the algorithm $P_{\gamma}=(P_{\mu}\xleftarrow{\,\mathrm{H}}l)$. This implies there is a $k$-box in row $r-1$ of $P_{\mu}$ (e.g. Fig.\ref{fig:11}b). According to the Hecke insertion algorithm, inserting some integer $x$ into the $(r-1)$-th row of $P_{\mu}$ will replace $k$ with $x$ and insert $k$ into the next row. Therefore, while $P_{\gamma}=(P_{\mu}\xleftarrow{\,\mathrm{H}}l)$ terminates in row $r$, $P_{\upsilon}=(P_{\lambda}\xleftarrow{\,\mathrm{H}}l)$ terminates in row $r-1$, with the terminating box being an $x$-box, and $P_{\upsilon}$ having one more box than $P_{\lambda}$ in row $r-1$.
	\\
	\item [(B16)] By the proof of (B15), we know that there is a $k$-box in both row $r$ and row $r-1$ of $P_{\mu}$. Since $\gamma_{r}=\upsilon_{r}+1$ and $\gamma_{r-1}=\upsilon_{r-1}+1$, there is also a $k$-box in both row $r$ and row $r-1$ of $P_{\gamma}$ (see Fig.\ref{fig:11}c). Since the edge between $\mu$ and $\gamma$ is labeled $r$, according to the Hecke insertion algorithm, inserting some integer $x$ into the $(r-1)$-th row of $P_{\mu}$ will produce an output integer $k$ without changing this row, and then insert $k$ into the next row. Therefore, while $P_{\gamma}=(P_{\mu}\xleftarrow{\,\mathrm{H}}l)$ terminates in row $r$, $P_{\upsilon}=(P_{\lambda}\xleftarrow{\,\mathrm{H}}l)$ terminates in row $r-1$, and $\lambda$ has the same shape as $\upsilon$, thus $\lambda=\upsilon$.
\end{proof}

\begin{figure}
	$$
	\Einheit.2cm
	\PfadDicke{.5pt}
	\Pfad(0,0),111111\endPfad
	\Pfad(3,3),111\endPfad
	\Pfad(3,6),111\endPfad
	\Pfad(3,0),222222\endPfad
	\Pfad(6,0),222222222\endPfad
	\Label\ro{k}(4,1)
	\Label\ro{x}(4,4)
	\Label\ro{=}(9,4)
	\Label\ro{P_{\gamma}}(3,-2)
	\hbox{\hskip2cm}
	\Einheit.2cm
	\PfadDicke{.5pt}
	\Pfad(0,0),111111\endPfad
	\Pfad(3,3),111\endPfad
	\Pfad(3,6),111\endPfad
	\Pfad(3,0),222222\endPfad
	\Pfad(6,0),222222222\endPfad
	\Label\ro{k}(4,1)
	\Label\ro{x}(4,4)
	\Label\ro{\xleftarrow{\,\mathrm{H}}l}(8,4)
	\Label\ro{P_{\mu}}(3,-2)
	\hbox{\hskip3cm}
	\Einheit.2cm
	\PfadDicke{.5pt}
	\Pfad(0,0),111111\endPfad
	\Pfad(3,3),111111111\endPfad
	\Pfad(9,6),111\endPfad
	\Pfad(3,0),222\endPfad
	\Pfad(6,0),222\endPfad
	\Pfad(9,3),222\endPfad
	\Pfad(12,3),222222\endPfad
	\Label\ro{k}(4,1)
	\Label\ro{x}(10,4)
	\Label\ro{=}(15,4)
	\Label\ro{P_{\gamma}}(6,-2)
	\hbox{\hskip2.6cm}
	\Einheit.2cm
	\PfadDicke{.5pt}
	\Pfad(0,0),111111\endPfad
	\Pfad(3,3),111111111\endPfad
	\Pfad(9,6),111\endPfad
	\Pfad(3,0),222\endPfad
	\Pfad(6,0),222\endPfad
	\Pfad(9,3),222\endPfad
	\Pfad(12,3),222222\endPfad
	\Label\ro{k}(4,1)
	\Label\ro{k}(10,4)
	\Label\ro{\xleftarrow{\,\mathrm{H}}l}(14,4)
	\Label\ro{P_{\mu}}(6,-2)
	\hbox{\hskip4cm}
	\Einheit.2cm
	\PfadDicke{.5pt}
	\Pfad(0,0),111111\endPfad
	\Pfad(3,3),111111111\endPfad
	\Pfad(9,6),111\endPfad
	\Pfad(3,0),222\endPfad
	\Pfad(6,0),222\endPfad
	\Pfad(9,3),222\endPfad
	\Pfad(12,3),222222\endPfad
	\Label\ro{k}(4,1)
	\Label\ro{k}(10,4)
	\Label\ro{=}(15,4)
	\Label\ro{P_{\gamma}}(6,-2)
	\hbox{\hskip2.6cm}
	\Einheit.2cm
	\PfadDicke{.5pt}
	\Pfad(0,0),111111\endPfad
	\Pfad(3,3),111111111\endPfad
	\Pfad(9,6),111\endPfad
	\Pfad(3,0),222\endPfad
	\Pfad(6,0),222\endPfad
	\Pfad(9,3),222\endPfad
	\Pfad(12,3),222222\endPfad
	\Label\ro{k}(4,1)
	\Label\ro{k}(10,4)
	\Label\ro{\xleftarrow{\,\mathrm{H}}l}(14,4)
	\Label\ro{P_{\mu}}(6,-2)
	\Label\ro{\leftarrow}${\tiny{\it the r-th row}}$(13,0.5)
	\hskip4cm
	$$
	\vspace{0.1cm}
	\centerline{\small a. An illustration of (B13)
		\hskip1.9cm
		b. An illustration of (B15)
		\hskip1.9cm
		c. An illustration of (B16)}
	\caption{Some illustrations for backward local rules.}
	\label{fig:11}
\end{figure}

In view of the Hecke growth diagram's forward local rules and backward local rules, we have the following theorem. Fig.\ref{fig:12} illustrates this.

\begin{theorem}\label{Th3.4}
	Let $R$ be a rectangle with $n$ columns and $m$ rows. The $01$-fillings of $R$ with the property that each column contains at most one $1$ are in bijection with sequences $(\emptyset=\mu_0,e_1,\mu_1,e_2,\mu_2,\ldots,\mu_{n-1},e_n,\mu_n=\upsilon_m,\upsilon_{m-1},\ldots,\upsilon_1,\upsilon_0=\emptyset)$. The sequence satisfies the following conditions:
	\begin{itemize}
		\item $\mu_i$ and $\upsilon_j$ are partitions for $0\leq i \leq n$ and $0\leq j \leq m$.
		
		\item For $1 \leq i \leq n$, $\mu_{i-1}\subseteq\mu_i$, and $\mu_{i-1}$ and $\mu_i$ differ by at most one box.
		
		\item For \( 1 \leq j \leq m \), \( \upsilon_{j-1} \subseteq \upsilon_j \), and \( \upsilon_j / \upsilon_{j-1} \) is a rook strip.
		
		\item For \( 1 \leq k \leq n \), \( e_k \) is a non-negative integer and $e_k$ is not greater than the number of rows of $u_k$. $\mu_{k-1}\neq \mu_k$ if $e_k=0$, while $\mu_{k-1}=\mu_k$ if $e_k>0$.
		
	\end{itemize}
\end{theorem}

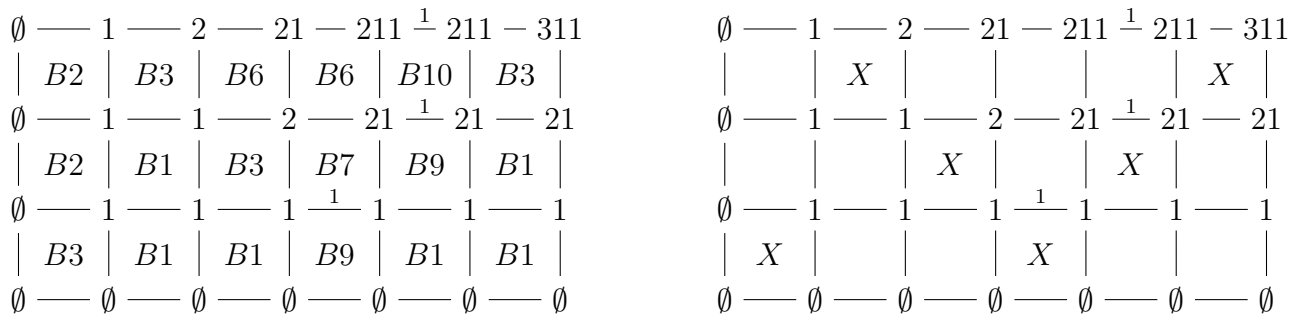
\begin{figure}
		$$
	\begin{tikzpicture}[scale=1.2]
		\node (00) at (0,0) {$\emptyset$};
		\node (10) at (1,0) {$\emptyset$};
		\node (20) at (2,0) {$\emptyset$};
		\node (30) at (3,0) {$\emptyset$};
		\node (40) at (4,0) {$\emptyset$};
		\node (50) at (5,0) {$\emptyset$};
		\node (01) at (0,1) {$\emptyset$};
		\node (11) at (1,1) {$1$};
		\node (21) at (2,1) {$1$};
		\node (31) at (3,1) {$1$};
		\node (41) at (4,1) {$1$};
		\node (51) at (5,1) {$1$};
		\node (02) at (0,2) {$\emptyset$};
		\node (12) at (1,2) {$1$};
		\node (22) at (2,2) {$1$};
		\node (32) at (3,2) {$2$};
		\node (42) at (4,2) {$21$};
		\node (52) at (5,2) {$21$};
		\node (03) at (0,3) {$\emptyset$};
		\node (13) at (1,3) {$1$};
		\node (23) at (2,3) {$2$};
		\node (33) at (3,3) {$21$};
		\node (43) at (4,3) {$211$};
		\node (53) at (5,3) {$211$};
		\node (60) at (6,0) {$\emptyset$};
		\node (61) at (6,1) {$1$};
		\node (62) at (6,2) {$21$};
		\node (63) at (6,3) {$311$};
		\node at (3.5,1.15) {\tiny $1$};
		\node at (4.5,2.15) {\tiny $1$};
		\node at (4.5,3.15) {\tiny $1$};
		\node at (0.5,0.5) {$B3$};
		\node at (0.5,1.5) {$B2$};
		\node at (0.5,2.5) {$B2$};
		\node at (1.5,0.5) {$B1$};
		\node at (1.5,1.5) {$B1$};
		\node at (1.5,2.5) {$B3$};
		\node at (2.5,0.5) {$B1$};
		\node at (2.5,1.5) {$B3$};
		\node at (2.5,2.5) {$B6$};
		\node at (3.5,0.5) {$B9$};
		\node at (3.5,1.5) {$B7$};
		\node at (3.5,2.5) {$B6$};
		\node at (4.5,0.5) {$B1$};
		\node at (4.5,1.5) {$B9$};
		\node at (4.5,2.5) {$B10$};
		\node at (5.5,0.5) {$B1$};
		\node at (5.5,1.5) {$B1$};
		\node at (5.5,2.5) {$B3$};

		\draw (00)--(10)--(20)--(30)--(40)--(50)--(60)
		(01)--(11)--(21)--(31)--(41)--(51)--(61)
		(02)--(12)--(22)--(32)--(42)--(52)--(62)
		(03)--(13)--(23)--(33)--(43)--(53)--(63)
		
		(00)--(01)--(02)--(03)
		(10)--(11)--(12)--(13)
		(20)--(21)--(22)--(23)
		(30)--(31)--(32)--(33)
		(40)--(41)--(42)--(43)
		(50)--(51)--(52)--(53)
		(60)--(61)--(62)--(63);

	\end{tikzpicture}
	\hbox{\hskip1.5cm}
	\begin{tikzpicture}[scale=1.2]
		\node (00) at (0,0) {$\emptyset$};
		\node (10) at (1,0) {$\emptyset$};
		\node (20) at (2,0) {$\emptyset$};
		\node (30) at (3,0) {$\emptyset$};
		\node (40) at (4,0) {$\emptyset$};
		\node (50) at (5,0) {$\emptyset$};
		\node (01) at (0,1) {$\emptyset$};
		\node (.51.5) at (0.5,0.5) {$X$};
		\node (11) at (1,1) {$1$};
		\node (21) at (2,1) {$1$};
		\node (31) at (3,1) {$1$};
		\node (41) at (4,1) {$1$};
		\node (51) at (5,1) {$1$};
		\node (02) at (0,2) {$\emptyset$};
		\node (12) at (1,2) {$1$};
		\node (22) at (2,2) {$1$};
		\node (32) at (3,2) {$2$};
		\node (1.50.5) at (1.5,2.5) {$X$};
		\node (42) at (4,2) {$21$};
		\node (52) at (5,2) {$21$};
		\node (03) at (0,3) {$\emptyset$};
		\node (13) at (1,3) {$1$};
		\node (23) at (2,3) {$2$};
		\node (33) at (3,3) {$21$};
		\node (43) at (4,3) {$211$};
		\node (53) at (5,3) {$211$};
		\node (2.52.5) at (2.5,1.5) {$X$};
		\node (3.52.5) at (3.5,0.5) {$X$};
		\node (4.50.5) at (4.5,1.5) {$X$};
		\node (5,51.5) at (5.5,2.5) {$X$};
		\node (60) at (6,0) {$\emptyset$};
		\node (61) at (6,1) {$1$};
		\node (62) at (6,2) {$21$};
		\node (63) at (6,3) {$311$};
		\node at (3.5,1.15) {\tiny $1$};
		\node at (4.5,2.15) {\tiny $1$};
		\node at (4.5,3.15) {\tiny $1$};

		\draw (00)--(10)--(20)--(30)--(40)--(50)--(60)
		(01)--(11)--(21)--(31)--(41)--(51)--(61)
		(02)--(12)--(22)--(32)--(42)--(52)--(62)
		(03)--(13)--(23)--(33)--(43)--(53)--(63)
		
		(00)--(01)--(02)--(03)
		(10)--(11)--(12)--(13)
		(20)--(21)--(22)--(23)
		(30)--(31)--(32)--(33)
		(40)--(41)--(42)--(43)
		(50)--(51)--(52)--(53)
		(60)--(61)--(62)--(63);

	\end{tikzpicture}
	$$
	\caption{The backward local rules applied at each step to reconstruct the Hecke growth diagram on the left and the resulting filling on the right (we suppress $0$'s for better visibility).}
	\label{fig:12}
\end{figure}

	

\vspace{0.5cm}

\section{Main results}

\subsection{A Jeu de taquin map for straight increasing tableaux}
$\newline$

In this section, we define a {\it jeu de taquin map} for straight increasing tableaux.

{\bf Case 1.} If $T$ is a standard straight increasing tableau, and $T'$ is the resulting tableau obtained from applying the jeu de taquin map to $T$, denoted as $jdt(T)$, then the process to obtain $T'$ is as follows:

\begin{itemize}
	\item [Step 1.] Let $C$ be the top-left box of $T$. Replace the entry $1$ in $C$ with a $\bullet$.
	\item [Step 2.] If the entry $2$ is directly below or to the right of $\bullet$, swap $\bullet$ and $2$. Repeat this process: 
	swap $\bullet$ with $3$, then with $4$, and so on, until the $\bullet$ has been swapped with all the entries in $T$.
	\item [Step 3.] Remove the boxes marked with $\bullet$ and decrease all entries by 1. The resulting tableau is $T'$.
\end{itemize}

Clearly, $T'$ is also a standard straight increasing tableau.

The mapping $jdt(T)=T'$ is not invertible on its own but becomes invertible if we know the shapes of $T$ and $T'$. If we know the shapes of $T$ and $T'$, then we denote this inverse by $T=jdt^{-1}(T')$. Let the shape of $T$ be $\gamma$ and the shape of $T'$ be $\mu$, the steps to reconstruct $T$ from $T'$ are as follows:

\begin{itemize}
	\item [Step 1.] Let $\gamma/\mu$ be the set of boxes. We denote this set of boxes by $C'$ and mark the boxes in $C'$ with $\bullet$.
	\item [Step 2.] Let $m$ be the maximum value of $T'$. If $m$ is directly above or to the left of $\bullet$, swap $\bullet$ and $m$. Repeat this process: swap $\bullet$ with $m-1$, then with $m-2$, and so on, until the $\bullet$ has been swapped with all the entries in $T'$. 	
	\item [Step 3.] Increase all entries by 1 and then replace the $\bullet$ with 1. The resulting tableau is $T$.
\end{itemize}

The example of the jeu de taquin map for standard straight increasing tableau $T$ is shown in Fig\ref{fig:15}, where $T\in$ SIT$^{st}$(4,3,1).

\begin{figure}
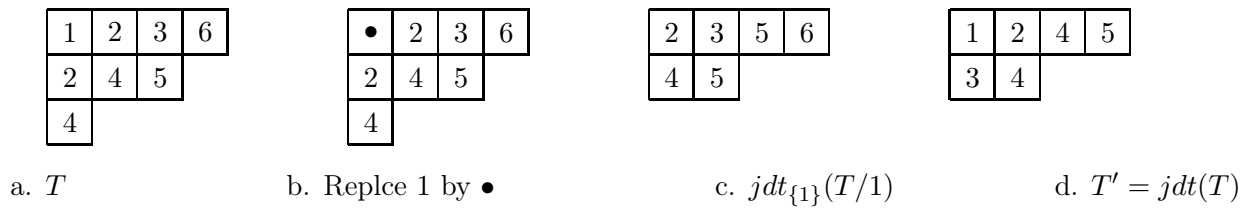

	$$
	\Einheit.2cm
	\PfadDicke{.5pt}
	\Pfad(12,6),222\endPfad
	\Pfad(9,3),222222\endPfad
	\Pfad(6,3),222222\endPfad
	\Pfad(3,0),222222222\endPfad
	\Pfad(0,0),222222222\endPfad
	\Pfad(0,9),111111111111\endPfad
	\Pfad(0,6),111111111111\endPfad
	\Pfad(0,3),111111111\endPfad
	\Pfad(0,0),111\endPfad
	\Label\ro{\text {\small$1$}}(1,7)
	\Label\ro{\text {\small$2$}}(1,4)
	\Label\ro{\text {\small$4$}}(1,1)
	\Label\ro{\text {\small$2$}}(4,7)
	\Label\ro{\text {\small$4$}}(4,4)
	\Label\ro{\text {\small$3$}}(7,7)
	\Label\ro{\text {\small$5$}}(7,4)
	\Label\ro{\text {\small$6$}}(10,7)
	\hbox{\hskip4cm}
	\Einheit.2cm
	\PfadDicke{.5pt}
	\Pfad(12,6),222\endPfad
	\Pfad(9,3),222222\endPfad
	\Pfad(6,3),222222\endPfad
	\Pfad(3,0),222222222\endPfad
	\Pfad(0,0),222222222\endPfad
	\Pfad(0,9),111111111111\endPfad
	\Pfad(0,6),111111111111\endPfad
	\Pfad(0,3),111111111\endPfad
	\Pfad(0,0),111\endPfad
	\Label\ro{\text {\small$\bullet$}}(1,7)
	\Label\ro{\text {\small$2$}}(1,4)
	\Label\ro{\text {\small$4$}}(1,1)
	\Label\ro{\text {\small$2$}}(4,7)
	\Label\ro{\text {\small$4$}}(4,4)
	\Label\ro{\text {\small$3$}}(7,7)
	\Label\ro{\text {\small$5$}}(7,4)
	\Label\ro{\text {\small$6$}}(10,7)
	\hbox{\hskip4cm}
	\Einheit.2cm
	\PfadDicke{.5pt}
	\Pfad(12,6),222\endPfad
	\Pfad(9,6),222\endPfad
	\Pfad(6,3),222222\endPfad
	\Pfad(3,3),222222\endPfad
	\Pfad(0,3),222222\endPfad
	\Pfad(0,9),111111111111\endPfad
	\Pfad(0,6),111111111111\endPfad
	\Pfad(0,3),111111\endPfad
	\Label\ro{\text {\small$2$}}(1,7)
	\Label\ro{\text {\small$4$}}(1,4)
	\Label\ro{\text {\small$3$}}(4,7)
	\Label\ro{\text {\small$5$}}(4,4)
	\Label\ro{\text {\small$5$}}(7,7)
	\Label\ro{\text {\small$6$}}(10,7)
	\hbox{\hskip4cm}
	\Einheit.2cm
	\PfadDicke{.5pt}
	\Pfad(12,6),222\endPfad
	\Pfad(9,6),222\endPfad
	\Pfad(6,3),222222\endPfad
	\Pfad(3,3),222222\endPfad
	\Pfad(0,3),222222\endPfad
	\Pfad(0,9),111111111111\endPfad
	\Pfad(0,6),111111111111\endPfad
	\Pfad(0,3),111111\endPfad
	\Label\ro{\text {\small$1$}}(1,7)
	\Label\ro{\text {\small$3$}}(1,4)
	\Label\ro{\text {\small$2$}}(4,7)
	\Label\ro{\text {\small$4$}}(4,4)
	\Label\ro{\text {\small$4$}}(7,7)
	\Label\ro{\text {\small$5$}}(10,7)
	\hskip3cm
	$$
	\centerline{\small 
		\hskip3cm
		a. $T$
		\hskip2.8cm
		b. Replce $1$ by $\bullet$
		\hskip2.8cm
		c. $jdt_{\{1\}}(T/1)$
		\hskip2cm
		d. $T^\prime=jdt(T)$
		\hskip2.5cm}
	\caption{An example from $T$ to $T^\prime$.}
	\label{fig:15}
\end{figure}

{\bf Case 2.} If $T$ is a straight increasing tableau but not a standard straight increasing tableau, this leads two subcases:

\begin{itemize}
	\item [(C1)] If the top-left box of $T$ has value $k$, where $k>1$.
	\item [(C2)] If the top left box of $T$ has value $1$, but the set of entries does not constitute an integer interval with left endpoint $1$.
\end{itemize}

For (C1), we define the rules of the jeu de taquin map for $T$ as follows: decrease all entries of $T$ by 1, resulting in the tableau $T'$. Thus, $T$ and $T'$ have the same shape. The mapping is clearly invertible, and $T$ can be reconstructed by increasing all entries of $T'$ by 1. An example is shown in Fig.\ref{fig:16}a,b.

For (C2), we use the same rules of the jeu de taquin map as those for standard straight increasing tableaux. An example is shown in Fig.\ref{fig:16}c,d.

\begin{figure}
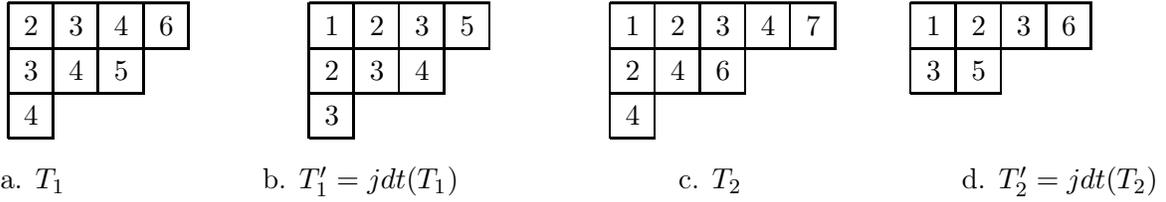

	$$
	\Einheit.2cm
	\PfadDicke{.5pt}
	\Pfad(12,6),222\endPfad
	\Pfad(9,3),222222\endPfad
	\Pfad(6,3),222222\endPfad
	\Pfad(3,0),222222222\endPfad
	\Pfad(0,0),222222222\endPfad
	\Pfad(0,9),111111111111\endPfad
	\Pfad(0,6),111111111111\endPfad
	\Pfad(0,3),111111111\endPfad
	\Pfad(0,0),111\endPfad
	\Label\ro{\text {\small$2$}}(1,7)
	\Label\ro{\text {\small$3$}}(1,4)
	\Label\ro{\text {\small$4$}}(1,1)
	\Label\ro{\text {\small$3$}}(4,7)
	\Label\ro{\text {\small$4$}}(4,4)
	\Label\ro{\text {\small$4$}}(7,7)
	\Label\ro{\text {\small$5$}}(7,4)
	\Label\ro{\text {\small$6$}}(10,7)
	\hbox{\hskip4cm}
	\Einheit.2cm
	\PfadDicke{.5pt}
	\Pfad(12,6),222\endPfad
	\Pfad(9,3),222222\endPfad
	\Pfad(6,3),222222\endPfad
	\Pfad(3,0),222222222\endPfad
	\Pfad(0,0),222222222\endPfad
	\Pfad(0,9),111111111111\endPfad
	\Pfad(0,6),111111111111\endPfad
	\Pfad(0,3),111111111\endPfad
	\Pfad(0,0),111\endPfad
	\Label\ro{\text {\small$1$}}(1,7)
	\Label\ro{\text {\small$2$}}(1,4)
	\Label\ro{\text {\small$3$}}(1,1)
	\Label\ro{\text {\small$2$}}(4,7)
	\Label\ro{\text {\small$3$}}(4,4)
	\Label\ro{\text {\small$3$}}(7,7)
	\Label\ro{\text {\small$4$}}(7,4)
	\Label\ro{\text {\small$5$}}(10,7)
	\hbox{\hskip4cm}
	\Einheit.2cm
	\PfadDicke{.5pt}
	\Pfad(15,6),222\endPfad
	\Pfad(12,6),222\endPfad
	\Pfad(9,3),222222\endPfad
	\Pfad(6,3),222222\endPfad
	\Pfad(3,0),222222222\endPfad
	\Pfad(0,0),222222222\endPfad
	\Pfad(0,9),111111111111111\endPfad
	\Pfad(0,6),111111111111111\endPfad
	\Pfad(0,3),111111111\endPfad
	\Pfad(0,0),111\endPfad
	\Label\ro{\text {\small$1$}}(1,7)
	\Label\ro{\text {\small$2$}}(1,4)
	\Label\ro{\text {\small$4$}}(1,1)
	\Label\ro{\text {\small$2$}}(4,7)
	\Label\ro{\text {\small$4$}}(4,4)
	\Label\ro{\text {\small$3$}}(7,7)
	\Label\ro{\text {\small$6$}}(7,4)
	\Label\ro{\text {\small$4$}}(10,7)
	\Label\ro{\text {\small$7$}}(13,7)
	\hbox{\hskip4cm}
	\Einheit.2cm
	\PfadDicke{.5pt}
	\Pfad(12,6),222\endPfad
	\Pfad(9,6),222\endPfad
	\Pfad(6,3),222222\endPfad
	\Pfad(3,3),222222\endPfad
	\Pfad(0,3),222222\endPfad
	\Pfad(0,9),111111111111\endPfad
	\Pfad(0,6),111111111111\endPfad
	\Pfad(0,3),111111\endPfad
	\Label\ro{\text {\small$1$}}(1,7)
	\Label\ro{\text {\small$3$}}(1,4)
	\Label\ro{\text {\small$2$}}(4,7)
	\Label\ro{\text {\small$5$}}(4,4)
	\Label\ro{\text {\small$3$}}(7,7)
	\Label\ro{\text {\small$6$}}(10,7)
	\hskip3cm
	$$
	\centerline{\small 
		\hskip2.8cm
		a. $T_1$
		\hskip2.5cm
		b. $T'_1=jdt(T_1)$
		\hskip2.8cm
		c. $T_2$
		\hskip2.8cm
		d. $T'_2=jdt(T_2)$
		\hskip2.5cm}
	\caption{Two examples of jeu de taquin for straight increasing tableaux.}
	\label{fig:16}
\end{figure}

Based on the above discussion, we have the following theorem.

\begin{theorem}
	Let $T$ and $T^{\prime}$ be straight increasing tableaux such that $jdt(T)=T^{\prime}$. Then the mapping is invertible if we know the shapes of $T$ and $T'$.
	\label{Th4.1}
\end{theorem}

\subsection{Hecke growth diagrams for $01$-fillings of stack polyominoes}

$\newline$

Let $\mathcal{M}$ be a stack polyomino with left-justified rows (cf. Fig.\ref{fig:1}b), and let $\mathcal{B}$ denote its right/down boundary. We define a $turn$ to be the top-left corner of $\mathcal{B}$.
For example, in Fig.\ref{fig:18} (where the labellings of the corners of its rigth/up boundary should be ignored at this point), there are three turns in this stack polyomino, marked with $\bullet$, and exactly three turns.

\begin{figure}
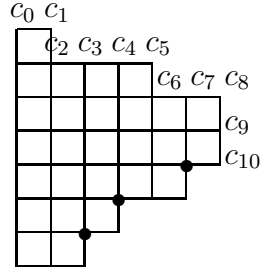

	$$
	\Einheit.15cm
	\PfadDicke{.5pt}
	\Pfad(0,0),111111\endPfad
	\Pfad(0,3),111111111\endPfad
	\Pfad(0,6),111111111111111\endPfad
	\Pfad(0,9),111111111111111111\endPfad
	\Pfad(0,12),111111111111111111\endPfad
	\Pfad(0,15),111111111111111111\endPfad
	\Pfad(0,18),111111111111\endPfad
	\Pfad(0,21),111\endPfad
	\Pfad(0,0),222222222222222222222\endPfad
	\Pfad(3,0),222222222222222222222\endPfad
	\Pfad(6,0),222222222222222222\endPfad
	\Pfad(9,3),222222222222222\endPfad
	\Pfad(12,6),222222222222\endPfad
	\Pfad(15,6),222222222\endPfad
	\Pfad(18,9),222222\endPfad
	\Label\ro{\text {$\bullet$}}(5.5,2.3)
	\Label\ro{\text {$\bullet$}}(8.5,5.3)
	\Label\ro{\text {$\bullet$}}(14.5,8.3)
	\Label\ro{\text {$c_0$}}(0,22)
	\Label\ro{\text {$c_1$}}(3,22)
	\Label\ro{\text {$c_2$}}(3,19)
	\Label\ro{\text {$c_3$}}(6,19)
	\Label\ro{\text {$c_4$}}(9,19)
	\Label\ro{\text {$c_5$}}(12,19)
	\Label\ro{\text {$c_6$}}(13,16)
	\Label\ro{\text {$c_7$}}(16,16)
	\Label\ro{\text {$c_8$}}(19,16)
	\Label\ro{\text {$c_9$}}(19,12)
    \Label\ro{\text {$c_{10}$}}(19.5,9)	
	\hskip2cm
	$$
	\caption{Some turns of a stack polyomino.}
	\label{fig:18}
\end{figure}

Given a $01$-filling $\pi$ of $\mathcal{M}$ such that there is at most one $1$ in each column, we obtain the {\it Hecke growth diagram} for $\pi$ of $\mathcal{M}$ by the following induction. We denote this diagram as $\mathcal{M}$$(\pi)$.

First, For the first column (from the left), we treat it as a rectangle and derive the sequence of partitions along the right border using {\it forward local rules} of Section 2.2. This sequence corresponds to a straight increasing tableau by Theorem \ref{PQ-w}.

Next, suppose we have obtained the sequence of partitions along the right border of some column $\mathcal{C}$. Let the right neighbor of $\mathcal{C}$ be column $\mathcal{L}$. Let the sequence of partitions along the right border of $\mathcal{C}$, which ranges from the bottom-right of $\mathcal{C}$ to the top-left of $\mathcal{L}$, correspond to a straight increasing tableau $P$ with shape $\lambda$.
We can then construct the sequence of partitions along the right border of $\mathcal{L}$ and the labels on the left edges of these partitions as follows:

\begin{itemize}
	\item if $\mathcal{L}$ is bottom-justified with $\mathcal{C}$, we can continue to construct the sequence of partitions along the right border of $\mathcal{L}$ and the labels on the left edges of these partitions by using the {\it forward local rules}. The top-left corner of $\mathcal{L}$ is labeled with partition $\lambda$. 
	An illustration can be seen in the left diagram of Fig.\ref{fig:induction}. 
	
	\item if $\mathcal{L}$ is not bottom-justified with $\mathcal{C}$, we perform the jeu de taquin map on $P$ as many times as the number of rows below this turn, resulting in a new straight increasing tableau $P^{\prime}$. Since a straight increasing tableau corresponds one by one to a sequence of partitions, the sequence of partitions corresponding to $P^{\prime}$ is then taken as the sequence of partitions along the left border of $\mathcal{L}$. We can then continue to construct the sequence of partitions along the right border of $\mathcal{L}$ and the labels by using forward local rules. The top-left corner of $\mathcal{L}$ is labeled with a tuple of partitions $(\lambda_0,\lambda_1,\dots,\lambda_l)$, where $l$ is the number of rows below this turn and $\lambda_l=\lambda$, $\lambda_{l-1}$ is the shape of $jdt(P)$, $\lambda_{l-2}$ is the shape of $jdt(jdt(P))$, and $\lambda_{l-i}$ is the shape of $\underbrace{jdt(\cdots jdt}_{i}(P))$, for $i\leq l$. An illustration can be seen in the right diagram of Fig.\ref{fig:induction}. 
\end{itemize}

Finally, we repeat the above method until we obtain the sequence of partitions along the right border of the last column of $\mathcal{M}$. It is easy to check that the sequence along the top-right border of $\mathcal{M}$ satisfies the conditions of the theorem.

\begin{figure}[h]
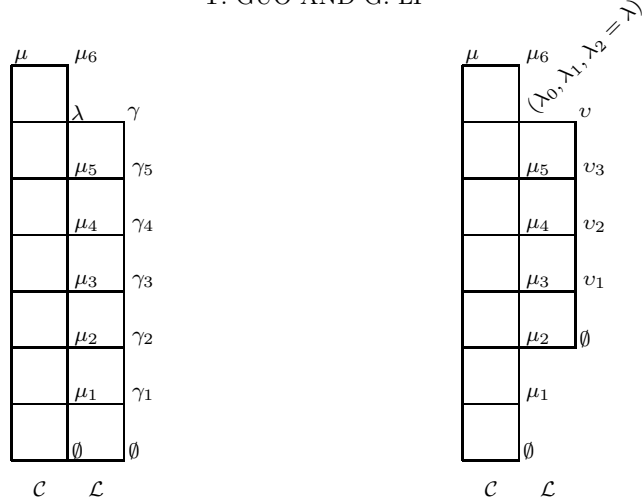

	$$
	\Einheit.25cm
	\PfadDicke{.5pt}
	\Pfad(0,0),111111\endPfad
	\Pfad(0,3),111111\endPfad
	\Pfad(0,6),111111\endPfad
	\Pfad(0,9),111111\endPfad
	\Pfad(0,12),111111\endPfad
	\Pfad(0,15),111111\endPfad
	\Pfad(0,18),111111\endPfad
	\Pfad(0,21),111\endPfad
	\Pfad(0,0),222222222222222222222\endPfad
	\Pfad(3,0),222222222222222222222\endPfad
	\Pfad(6,0),222222222222222222\endPfad
	\Label\ro{\text {\tiny $\emptyset$}}(3,0)
	\Label\ro{\text {\tiny $\mu_1$}}(3.5,3)
	\Label\ro{\text {\tiny $\mu_2$}}(3.5,6)
	\Label\ro{\text {\tiny $\mu_3$}}(3.5,9)
	\Label\ro{\text {\tiny $\mu_4$}}(3.5,12)
	\Label\ro{\text {\tiny $\mu_5$}}(3.5,15)
	\Label\ro{\text {\tiny $\lambda$}}(3,18)
	\Label\ro{\text {\tiny $\mu_6$}}(3.5,21)
	\Label\ro{\text {\tiny $\mu$}}(0,21)
	\Label\ro{\text {\tiny $\mathcal{C}$}}(1,-2)
	\Label\ro{\text {\tiny $\mathcal{L}$}}(4,-2)
	\Label\ro{\text {\tiny $\emptyset$}}(6,0)
	\Label\ro{\text {\tiny $\gamma_1$}}(6.5,3)
	\Label\ro{\text {\tiny $\gamma_2$}}(6.5,6)
	\Label\ro{\text {\tiny $\gamma_3$}}(6.5,9)
	\Label\ro{\text {\tiny $\gamma_4$}}(6.5,12)
	\Label\ro{\text {\tiny $\gamma_5$}}(6.5,15)
	\Label\ro{\text {\tiny $\gamma$}}(6,18)
	\hbox{\hskip6cm}
	\Einheit.25cm
	\PfadDicke{.5pt}
	\Pfad(0,0),111\endPfad
	\Pfad(0,3),111\endPfad
	\Pfad(0,6),111111\endPfad
	\Pfad(0,9),111111\endPfad
	\Pfad(0,12),111111\endPfad
	\Pfad(0,15),111111\endPfad
	\Pfad(0,18),111111\endPfad
	\Pfad(0,21),111\endPfad
	\Pfad(0,0),222222222222222222222\endPfad
	\Pfad(3,0),222222222222222222222\endPfad
	\Pfad(6,6),222222222222\endPfad
	\Label\ro{\text {\tiny $\emptyset$}}(3,0)
	\Label\ro{\text {\tiny $\mu_1$}}(3.5,3)
	\Label\ro{\text {\tiny $\mu_2$}}(3.5,6)
	\Label\ro{\text {\tiny $\mu_3$}}(3.5,9)
	\Label\ro{\text {\tiny $\mu_4$}}(3.5,12)
	\Label\ro{\text {\tiny $\mu_5$}}(3.5,15)
	\Label\ro{\text {\tiny $\mu$}}(0,21)
	\Label\ro{\slanttext {\tiny$(\lambda_0,\lambda_1,\lambda_2=\lambda)$}}(6,21)
	\Label\ro{\text {\tiny $\mu_6$}}(3.5,21)
	\Label\ro{\text {\tiny $\mathcal{C}$}}(1,-2)
	\Label\ro{\text {\tiny $\mathcal{L}$}}(4,-2)
	\Label\ro{\text {\tiny $\emptyset$}}(6,6)
	\Label\ro{\text {\tiny $\upsilon_1$}}(6.5,9)
	\Label\ro{\text {\tiny $\upsilon_2$}}(6.5,12)
	\Label\ro{\text {\tiny $\upsilon_3$}}(6.5,15)
	\Label\ro{\text {\tiny $\upsilon$}}(6,18)
	\hskip2cm
	$$
	\caption{An illustration for the Hecke growth diagram of stack polyominoes.}
	\label{fig:induction}
\end{figure}

Note that, if the top-right corner of $\mathcal{L}$ is labeled $(\upsilon_{0},\upsilon_{1},\dots,\upsilon_{m})$, the top-left corner of $\mathcal{L}$ is labeled $(\lambda_{0},\lambda_{1},\dots,\lambda_{l})$, and $\upsilon_{m}=\lambda_{0}$, then there exists a positive integer $r$ to record the row where the box terminates. We place $r$ at the top-right corner of $\upsilon_{m}$, denoted as $\upsilon_{m}^{r}$. Note that $r$ must not exceed the number of rows of $\upsilon_{m}$. 
In Fig.\ref{Hecke growth diagram}, this is a Hecke growth diagram for a $01$-filling of a stack polyomino, and the filling satisfies at most one $1$ in each column.

\subsection{The first main theorem}
$\newline$

Let $\mathcal{M}$ be a stack polyomino, we denote the right/up boundary of $\mathcal{M}$ by a sequence of corners $(c_{0},c_{1},\dots,c_{n})$, read from the left-top corner to the right-bottom corner. See
Fig.\ref{fig:18} for an example, where the maximally northeast border is $(c_{0},c_{1},\dots,c_{10})$.
The first main theorem of this paper is as follows.

\begin{theorem}
	Let $\mathcal{M}$ be a stack polyomino with the right/up boundary $(c_{0},c_{1},\dots,c_{n})$. The $01$-fillings of $\mathcal{M}$ with the property that each column contains at most one $1$ are in bijection with the sequence $(\emptyset=l_0,l_1,\dots,l_{n}=\emptyset)$ that satisfy the following conditions:
	\begin{itemize}
		\item  For $i=0,\dots,n$, $l_i$ is a label of $c_i$, and $l_i$ is either a partition or a sequence of partitions.
		
		\item  A corner is labeled by a sequence of partitions if and only if there is a turn below it and a corner directly to its right. Otherwise, the label of this corner is a partition.
		
		\item  A corner labeled by a sequence of partitions $(\lambda_{0},\lambda_{1},\dots,\lambda_{k})$ has the property that $\lambda_i/\lambda_{i-1}$ are the out corners of $\lambda_{i-1}$ for $1\leq i\leq k$.
		
		\item For a corner labeled by a sequence of partitions $(\lambda_{0},\lambda_{1},\dots,\lambda_{k})$ and its right neighbor $(\upsilon_{0},\upsilon_{1},\dots,\upsilon_{m})$, $\lambda_{0}$ and $\upsilon_{m}$ differ by at most one box. If $\lambda_{0}=\upsilon_{m}$, 
		then there probably exists a positive integer $r$ which does not exceed the number of rows of $\upsilon_{m}$, that is attached to $\upsilon_{m}$, we denote this as $\upsilon_{m}^r$.

		\item For the corner labeled by $(\mu_{0},\mu_{1},\dots,\mu_{j})$  and its top neighbor labeled by $\gamma$, $\gamma/\mu_{j}$ is a rook strip.
	\end{itemize}
	\label{Th4.3}
\end{theorem}

\begin{proof}

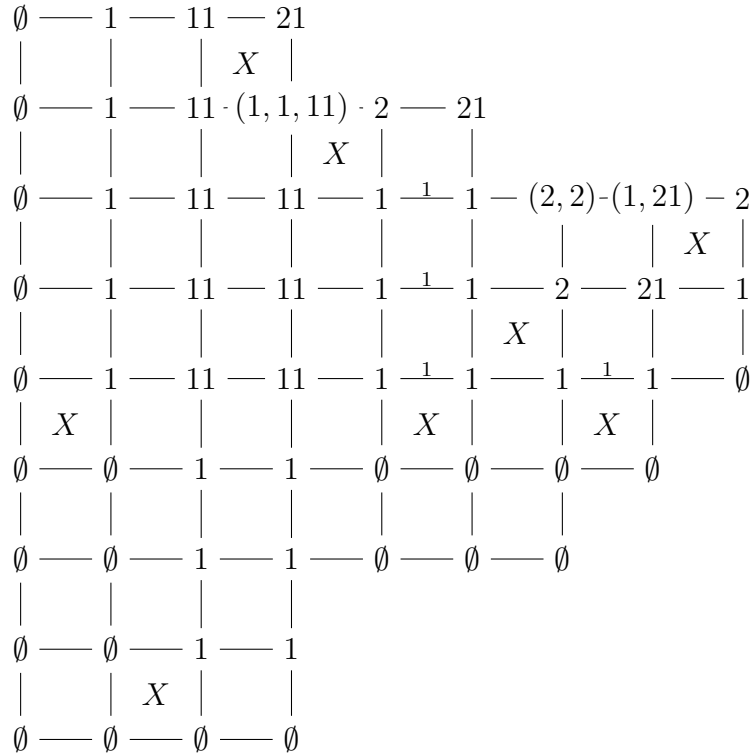
\begin{figure}[b]
	$$
	\begin{tikzpicture}[scale=1.2]
		\node (00) at (0,0) {$\emptyset$};
		\node (10) at (1,0) {$\emptyset$};
		\node (20) at (2,0) {$\emptyset$};
		\node (30) at (3,0) {$\emptyset$};
		\node (01) at (0,1) {$\emptyset$};
		\node (11) at (1,1) {$\emptyset$};
		\node (21) at (2,1) {$1$};
		\node (31) at (3,1) {$1$};
		\node (02) at (0,2) {$\emptyset$};
		\node (12) at (1,2) {$\emptyset$};
		\node (22) at (2,2) {$1$};
		\node (32) at (3,2) {$1$};
		\node (42) at (4,2) {$\emptyset$};
		\node (52) at (5,2) {$\emptyset$};
		\node (62) at (6,2) {$\emptyset$};
		\node (03) at (0,3) {$\emptyset$};
		\node (13) at (1,3) {$\emptyset$};
		\node (23) at (2,3) {$1$};
		\node (33) at (3,3) {$1$};
		\node (43) at (4,3) {$\emptyset$};
		\node (53) at (5,3) {$\emptyset$};
		\node (63) at (6,3) {$\emptyset$};
		\node (73) at (7,3) {$\emptyset$};
		\node (04) at (0,4) {$\emptyset$};
		\node (14) at (1,4) {$1$};
		\node (24) at (2,4) {$11$};
		\node (34) at (3,4) {$11$};
		\node (44) at (4,4) {$1$};
		\node (54) at (5,4) {$1$};
		\node (64) at (6,4) {$1$};
		\node (74) at (7,4) {$1$};
		\node (84) at (8,4) {$\emptyset$};
		
		\node (05) at (0,5) {$\emptyset$};
		\node (15) at (1,5) {$1$};
		\node (25) at (2,5) {$11$};
		\node (35) at (3,5) {$11$};
		\node (45) at (4,5) {$1$};
		\node (55) at (5,5) {$1$};
		\node (65) at (6,5) {$2$};
		\node (75) at (7,5) {$21$};
		\node (85) at (8,5) {$1$};
		
		\node (06) at (0,6) {$\emptyset$};
		\node (16) at (1,6) {$1$};
		\node (26) at (2,6) {$11$};
		\node (36) at (3,6) {$11$};
		\node (46) at (4,6) {$1$};
		\node (56) at (5,6) {$1$};
		\node (66) at (6,6) {$(2,2)$};
		\node (76) at (7,6) {$(1,21)$};
		\node (86) at (8,6) {$2$};
		
		\node (07) at (0,7) {$\emptyset$};
		\node (17) at (1,7) {$1$};
		\node (27) at (2,7) {$11$};
		\node (37) at (3,7) {$(1,1,11)$};
		\node (47) at (4,7) {$2$};
		\node (57) at (5,7) {$21$};
		
		\node (08) at (0,8) {$\emptyset$};
		\node (18) at (1,8) {$1$};
		\node (28) at (2,8) {$11$};
		\node (38) at (3,8) {$21$};
		
		\node at (0.5,3.5) {$X$};
		\node at (1.5,0.5) {$X$};
		\node at (2.5,7.5) {$X$};
		\node at (3.5,6.5) {$X$};
		\node at (4.5,3.5) {$X$};
		\node at (5.5,4.5) {$X$};
		\node at (6.5,3.5) {$X$};
		\node at (7.5,5.5) {$X$};
		
		\node at (4.5,4.1) {\tiny $1$};
		\node at (4.5,5.1) {\tiny $1$};
		\node at (4.5,6.1) {\tiny $1$};
		\node at (6.5,4.1) {\tiny $1$};

		\draw (00)--(10)--(20)--(30)
		(01)--(11)--(21)--(31)
		(02)--(12)--(22)--(32)--(42)--(52)--(62)
		(03)--(13)--(23)--(33)--(43)--(53)--(63)--(73)
		(04)--(14)--(24)--(34)--(44)--(54)--(64)--(74)--(84)
		(05)--(15)--(25)--(35)--(45)--(55)--(65)--(75)--(85)
		(06)--(16)--(26)--(36)--(46)--(56)--(66)--(76)--(86)
		(07)--(17)--(27)--(37)--(47)--(57)
		(08)--(18)--(28)--(38)
		
		(00)--(01)--(02)--(03)--(04)--(05)--(06)--(07)--(08)
		(10)--(11)--(12)--(13)--(14)--(15)--(16)--(17)--(18)
		(20)--(21)--(22)--(23)--(24)--(25)--(26)--(27)--(28)
		(30)--(31)--(32)--(33)--(34)--(35)--(36)--(37)--(38)
		(42)--(43)--(44)--(45)--(46)--(47)
		(52)--(53)--(54)--(55)--(56)--(57)
		(62)--(63)--(64)--(65)--(66)
		(73)--(74)--(75)--(76)
		(84)--(85)--(86);

	\end{tikzpicture}
	$$
	\caption{An example for the Hecke growth diagram of a stack polyomino.}
	\label{Hecke growth diagram}
\end{figure}

Let $\pi$ be a $01$-filling of $\mathcal{M}$ with at most one $1$ in each column, we can obtain the Hecke growth diagram of 
$\mathcal{M}(\pi)$, and it is easy to check that the sequence along the top-right border of $\mathcal{M}$ satisfies the conditions of the theorem.

    Now, given the sequence $(\emptyset=l_0,l_1,\dots,l_{n}=\emptyset)$ along the top-right border of $\mathcal{M}$ that satisfies the conditions in the theorem. Suppose we have obtained the sequence of partitions along the left border of some column $\mathcal{L}$ and the filling of $\mathcal{L}$, 
	Our goal is to reconstruct the filling and the sequence of partitions along the left border of its adjacent left column $\mathcal{C}$. We will discuss this in two cases.
	
	\begin{itemize}
		\item If these two columns are bottom-justified, as illustrated in the left diagram of Fig.\ref{fig:Th4.3}, and we have obtained the sequence of partitions $(\emptyset,\mu_1,\dots,\mu_k,\lambda)$ along the left border of column $\mathcal{L}$, we can use the backward local rules of Section 3 to reconstruct the filling and the sequence of partitions along the left border of $\mathcal{C}$ from $(\emptyset,\mu_1,\dots,\mu_k,\lambda,\dots,\mu)$ and $\rho$.
		
		\item If these two columns are not bottom-justified, as shown in the right diagram of Fig.\ref{fig:Th4.3}, and we have obtained the sequence of partitions $(\emptyset,\gamma_1,\dots,\gamma_l,\lambda_0)$ along the left border of column $\mathcal{L}$, this sequence corresponds to a straight increasing tableau $T_0$ with shape $\lambda_0$. Based on the sequence of partitions $(\lambda_0,\lambda_1,\dots,\lambda_m=\lambda)$, we perform the inverse jeu de taquin map on $T_0$ for $m$ times to obtain a straight increasing tableau $T_m$, which corresponds to the sequence of partitions $(\emptyset,\dots,\mu_m,\dots,\mu_k,\lambda)$. Using the backward local rules, we can then reconstruct the filling and the the sequence of partitions along the left border of $\mathcal{C}$ from $(\emptyset,\mu_1,\dots,\mu_k,\lambda,\dots,\mu)$ and $\rho$.
	\end{itemize}
	
	We repeat the above method until we complete the first column (from the left) of $\mathcal{M}$. Let $\omega$ be the filling we have obtained by this process. It is easy to check that the sequence along the top-right border of $\mathcal{M}(\omega)$ is equal to the sequence $(\emptyset=l_0,l_1,\dots,l_{n}=\emptyset)$.

	\begin{figure}
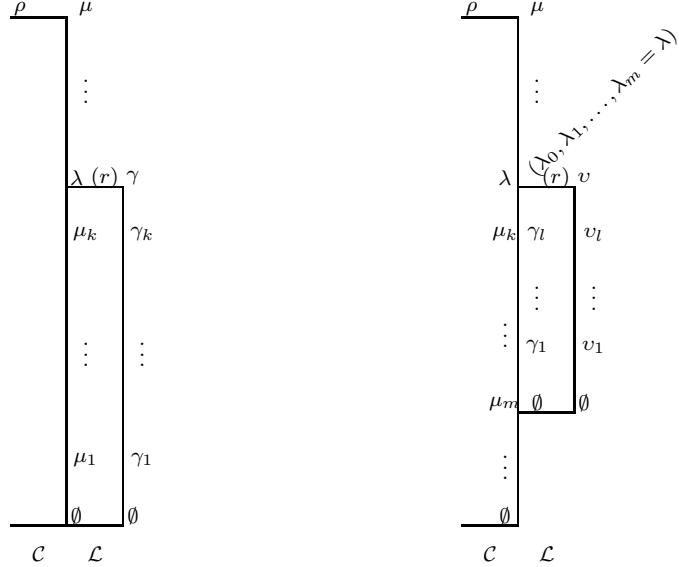

		$$
		\Einheit.25cm
		\PfadDicke{.5pt}
		\Pfad(0,0),111111\endPfad
		\Pfad(3,18),111\endPfad
		\Pfad(0,27),111\endPfad
		\Pfad(3,0),222222222222222222222222222\endPfad
		\Pfad(6,0),222222222222222222\endPfad
		\Label\ro{\text {\tiny $\emptyset$}}(3,0)
		\Label\ro{\text {\tiny $\mu_1$}}(3.5,3)
		\Label\ro{\text {\tiny $\vdots$}}(3.5,9)
		\Label\ro{\text {\tiny $\vdots$}}(3.5,23)
		\Label\ro{\text {\tiny $\mu_k$}}(3.5,15)
		\Label\ro{\text {\tiny $\lambda$}}(3,18)
		\Label\ro{\text {\tiny $\mu$}}(3.5,27)
		\Label\ro{\text {\tiny $\rho$}}(0,27)
		\Label\ro{\text {\tiny $\mathcal{C}$}}(1,-2)
		\Label\ro{\text {\tiny $\mathcal{L}$}}(4,-2)
		\Label\ro{\text {\tiny $\emptyset$}}(6,0)
		\Label\ro{\text {\tiny $\gamma_1$}}(6.5,3)
		\Label\ro{\text {\tiny $\vdots$}}(6.5,9)
		\Label\ro{\text {\tiny $\gamma_k$}}(6.5,15)
		\Label\ro{\text {\tiny $\gamma$}}(6,18)
		\Label\ro{\text {\tiny $(r)$}}(4.5,18)
		\hbox{\hskip6cm}
		\Einheit.25cm
		\PfadDicke{.5pt}
		\Pfad(0,0),111\endPfad
		\Pfad(3,6),111\endPfad
		\Pfad(3,18),111\endPfad
		\Pfad(0,27),111\endPfad
		\Pfad(3,0),222222222222222222222222222\endPfad
		\Pfad(6,6),222222222222\endPfad
		\Label\ro{\text {\tiny $\emptyset$}}(3.5,6)
		\Label\ro{\text {\tiny $\gamma_1$}}(3.5,9)
		\Label\ro{\text {\tiny $\vdots$}}(3.5,12)
		\Label\ro{\text {\tiny $\gamma_l$}}(3.5,15)
		\Label\ro{\text {\tiny $\rho$}}(0,27)
		\Label\ro{\slanttext {\tiny$(\lambda_0,\lambda_1,\dots,\lambda_m=\lambda)$}}(7,22)
		\Label\ro{\text {\tiny $\mu$}}(3.5,27)
		\Label\ro{\text {\tiny $\mathcal{C}$}}(1,-2)
		\Label\ro{\text {\tiny $\mathcal{L}$}}(4,-2)
		\Label\ro{\text {\tiny $\emptyset$}}(6,6)
		\Label\ro{\text {\tiny $\upsilon_1$}}(6.5,9)
		\Label\ro{\text {\tiny $\vdots$}}(6.5,12)
		\Label\ro{\text {\tiny $\upsilon_l$}}(6.5,15)
		\Label\ro{\text {\tiny $\upsilon$}}(6,18)
		\Label\ro{\text {\tiny $\vdots$}}(3.5,23)
		\Label\ro{\text {\tiny $\emptyset$}}(1.8,0)
		\Label\ro{\text {\tiny $\vdots$}}(1.8,3)
		\Label\ro{\text {\tiny $\mu_m$}}(1.8,6)
		\Label\ro{\text {\tiny $\vdots$}}(1.8,10)
		\Label\ro{\text {\tiny $\mu_k$}}(1.8,15)
		\Label\ro{\text {\tiny $\lambda$}}(1.8,18)
		\Label\ro{\text {\tiny $(r)$}}(4.5,18)
		\hskip2cm
		$$
		\caption{An illustration for the proof of Theorem \ref{Th4.3}.}
		\label{fig:Th4.3}
	\end{figure}

\end{proof}

See Fig.\ref{fig:20} for an example. This example demonstrates the correspondence between the $01$-filling of the stack polyomino and the sequence along the top-right border of the stack polyomino.

\begin{figure}
	$$
	\Einheit0.22cm
	\PfadDicke{0.5pt}
	\Pfad(0,0),111111111111111111\endPfad
	\Pfad(0,3),111111111111111111111111111111\endPfad
	\Pfad(0,6),111111111111111111111111111111\endPfad
	\Pfad(0,9),111111111111111111111111111111\endPfad
	\Pfad(0,12),111111111111111111111111111111111\endPfad
	\Pfad(0,15),111111111111111111111111111111111\endPfad
	\Pfad(0,18),111111111111111111111111111111111111111\endPfad
	\Pfad(0,21),111111111111111111111111111111111111111\endPfad
	\Pfad(0,24),111111111111111111111111111111111111111\endPfad
	\Pfad(0,27),111111111111111111111111111111111111\endPfad
	\Pfad(0,30),111111111111111111111111111111111111\endPfad
	\Pfad(0,0),222222222222222222222222222222\endPfad
	\Pfad(3,0),222222222222222222222222222222\endPfad
	\Pfad(6,0),222222222222222222222222222222\endPfad
	\Pfad(9,0),222222222222222222222222222222\endPfad
	\Pfad(12,0),222222222222222222222222222222\endPfad
	\Pfad(15,0),222222222222222222222222222222\endPfad
	\Pfad(18,0),222222222222222222222222222222\endPfad
	\Pfad(21,3),222222222222222222222222222\endPfad
	\Pfad(24,3),222222222222222222222222222\endPfad
	\Pfad(27,3),222222222222222222222222222\endPfad
	\Pfad(30,3),222222222222222222222222222\endPfad
	\Pfad(33,12),222222222222222222\endPfad
	\Pfad(36,18),222222222222\endPfad
	\Pfad(39,18),222222\endPfad
	\Label\ro{\text {\small$X$}}(1,16)
	\Label\ro{\text {\small$X$}}(4,7)
	\Label\ro{\text {\small$X$}}(7,10)
	\Label\ro{\text {\small$X$}}(10,19)
	\Label\ro{\text {\small$X$}}(13,25)
	\Label\ro{\text {\small$X$}}(16,1)
	\Label\ro{\text {\small$X$}}(19,7)
	\Label\ro{\text {\small$X$}}(22,16)
	\Label\ro{\text {\small$X$}}(25,22)
	\Label\ro{\text {\small$X$}}(28,13)
	\Label\ro{\text {\small$X$}}(31,13)
	\Label\ro{\text {\small$X$}}(34,28)
	\Label\ro{\text {\small$X$}}(37,19)
	\Label\ro{\text {\small$\emptyset$}}(-0.5,31)
	\Label\ro{\text {\small$1$}}(2.5,31)
	\Label\ro{\text {\small$11$}}(5.5,31)
	\Label\ro{\text {\small$21$}}(8.5,31)
	\Label\ro{\text {\small$31$}}(11.5,31)
	\Label\ro{\text {\small$41$}}(14.5,31)
	\Label\ro{\slanttext {\small$(41,411)$}}(19,33)
	\Label\ro{\text {\small$411$}}(20.5,31)
	\Label\ro{\text {\small$421$}}(23.5,31)
	\Label\ro{\text {\small$431$}}(26.5,31)
	\Label\ro{\slanttext {\small$(321,331,432,432)$}}(34,36)
	\Label\ro{\slanttext {\small$(21,32,321^3)$}}(35.5,34)
	\Label\ro{\text {\small$31$}}(36,31)
	\Label\ro{\text {\small$21$}}(37,27)
	\Label\ro{\text {\small$2$}}(36.5,24.5)
	\Label\ro{\text {\small$21$}}(40,24)
	\Label\ro{\text {\small$1$}}(40,21)
	\Label\ro{\text {\small$\emptyset$}}(40,18)
	\hskip7cm
	$$
	\vspace{1cm}
	$$
	\Einheit0.22cm
	\PfadDicke{0.5pt}
	\Pfad(0,0),111111111111111111\endPfad
	\Pfad(0,3),111111111111111111\endPfad
	\Pfad(0,6),111111111111111111\endPfad
	\Pfad(0,9),111111111111111111\endPfad
	\Pfad(0,12),111111111111111111\endPfad
	\Pfad(0,15),111111111111111111\endPfad
	\Pfad(0,18),111111111111111111\endPfad
	\Pfad(0,21),111111111111111111\endPfad
	\Pfad(0,24),111111111111111111\endPfad
	\Pfad(0,27),111111111111111111\endPfad
	\Pfad(0,30),111111111111111111\endPfad
	\Pfad(0,0),222222222222222222222222222222\endPfad
	\Pfad(3,0),222222222222222222222222222222\endPfad
	\Pfad(6,0),222222222222222222222222222222\endPfad
	\Pfad(9,0),222222222222222222222222222222\endPfad
	\Pfad(12,0),222222222222222222222222222222\endPfad
	\Pfad(15,0),222222222222222222222222222222\endPfad
	\Pfad(18,0),222222222222222222222222222222\endPfad
	\Pfad(27,3),111111111111\endPfad
	\Pfad(27,6),111111111111\endPfad
	\Pfad(27,9),111111111111\endPfad
	\Pfad(27,12),111111111111\endPfad
	\Pfad(27,15),111111111111\endPfad
	\Pfad(27,18),111111111111\endPfad
	\Pfad(27,21),111111111111\endPfad
	\Pfad(27,24),111111111111\endPfad
	\Pfad(27,27),111111111111\endPfad
	\Pfad(27,30),111111111111\endPfad
	\Pfad(27,3),222222222222222222222222222\endPfad
	\Pfad(30,3),222222222222222222222222222\endPfad
	\Pfad(33,3),222222222222222222222222222\endPfad
	\Pfad(36,3),222222222222222222222222222\endPfad
	\Pfad(39,3),222222222222222222222222222\endPfad
	\Pfad(48,12),111\endPfad
	\Pfad(48,15),111\endPfad
	\Pfad(48,18),111\endPfad
	\Pfad(48,21),111\endPfad
	\Pfad(48,24),111\endPfad
	\Pfad(48,27),111\endPfad
	\Pfad(48,30),111\endPfad
	\Pfad(48,12),222222222222222222\endPfad
	\Pfad(51,12),222222222222222222\endPfad
	\Pfad(60,18),111111\endPfad
	\Pfad(60,21),111111\endPfad
	\Pfad(60,24),111111\endPfad
	\Pfad(60,27),111\endPfad
	\Pfad(60,30),111\endPfad
	\Pfad(60,18),222222222222\endPfad
	\Pfad(63,18),222222222222\endPfad
	\Pfad(66,18),222222\endPfad
	\Label\ro{\text {\small$X$}}(1,16)
	\Label\ro{\text {\small$X$}}(4,7)
	\Label\ro{\text {\small$X$}}(7,10)
	\Label\ro{\text {\small$X$}}(10,19)
	\Label\ro{\text {\small$X$}}(13,25)
	\Label\ro{\text {\small$X$}}(16,1)
	\Label\ro{\text {\small$X$}}(28,7)
	\Label\ro{\text {\small$X$}}(31,16)
	\Label\ro{\text {\small$X$}}(34,22)
	\Label\ro{\text {\small$X$}}(37,13)
	\Label\ro{\text {\small$X$}}(49,13)
	\Label\ro{\text {\small$X$}}(61,28)
	\Label\ro{\text {\small$X$}}(64,19)
	\Label\ro{\text {\small$\emptyset$}}(-0.5,31)
	\Label\ro{\text {\small$1$}}(2.5,31)
	\Label\ro{\text {\small$11$}}(5.5,31)
	\Label\ro{\text {\small$21$}}(8.5,31)
	\Label\ro{\text {\small$31$}}(11.5,31)
	\Label\ro{\text {\small$41$}}(14.5,31)
	\Label\ro{\text {\small$411$}}(18,31)
	\Label\ro{\text {\small$\emptyset$}}(18.5,-0.5)
	\Label\ro{\text {\small$1$}}(18.5,2.5)
	\Label\ro{\text {\small$1$}}(18.5,5.5)
	\Label\ro{\text {\small$11$}}(18.5,8.5)
	\Label\ro{\text {\small$21$}}(18.5,11.5)
	\Label\ro{\text {\small$21$}}(18.5,14.5)
	\Label\ro{\text {\small$211$}}(19,17.5)
	\Label\ro{\text {\small$311$}}(19,20.5)
	\Label\ro{\text {\small$311$}}(19,23.5)
	\Label\ro{\text {\small$411$}}(19,26.5)
	\Label\ro{\text {\small$\emptyset$}}(25.5,2.5)
	\Label\ro{\text {\small$\emptyset$}}(25.5,5.5)
	\Label\ro{\text {\small$1$}}(25.5,8.5)
	\Label\ro{\text {\small$2$}}(25.5,11.5)
	\Label\ro{\text {\small$2$}}(25.5,14.5)
	\Label\ro{\text {\small$21$}}(25.5,17.5)
	\Label\ro{\text {\small$31$}}(25.5,20.5)
	\Label\ro{\text {\small$31$}}(25.5,23.5)
	\Label\ro{\text {\small$41$}}(25.5,26.5)
	\Label\ro{\text {\small$41$}}(25.5,31)
	\Label\ro{\text {\small$411$}}(29,31)
	\Label\ro{\text {\small$421$}}(32,31)
	\Label\ro{\text {\small$431$}}(35,31)
	\Label\ro{\text {\small$432$}}(39,31)
	\Label\ro{\text {\small$\emptyset$}}(39.5,2.5)
	\Label\ro{\text {\small$\emptyset$}}(39.5,5.5)
	\Label\ro{\text {\small$1$}}(39.5,8.5)
	\Label\ro{\text {\small$21$}}(39.5,11.5)
	\Label\ro{\text {\small$31$}}(39.5,14.5)
	\Label\ro{\text {\small$321$}}(40,17.5)
	\Label\ro{\text {\small$322$}}(40,20.5)
	\Label\ro{\text {\small$422$}}(40,23.5)
	\Label\ro{\text {\small$432$}}(40,26.5)
	\Label\ro{\text {\small$\emptyset$}}(47,11.5)
	\Label\ro{\text {\small$1$}}(47,14.5)
	\Label\ro{\text {\small$11$}}(46.5,17.5)
	\Label\ro{\text {\small$211$}}(46,20.5)
	\Label\ro{\text {\small$311$}}(46,23.5)
	\Label\ro{\text {\small$321$}}(46,26.5)
	\Label\ro{\text {\small$321$}}(46.5,31)
	\Label\ro{\text {\tiny$3$}}(49,30)
	\Label\ro{\text {\tiny$3$}}(49,27)
	\Label\ro{\text {\tiny$3$}}(49,24)
	\Label\ro{\text {\tiny$3$}}(49,21)
	\Label\ro{\text {\tiny$2$}}(49,18)
	\Label\ro{\text {\tiny$1$}}(49,15)
	\Label\ro{\text {\small$\emptyset$}}(51,11.5)
	\Label\ro{\text {\small$1$}}(51,14.5)
	\Label\ro{\text {\small$11$}}(51.5,17.5)
	\Label\ro{\text {\small$211$}}(52,20.5)
	\Label\ro{\text {\small$311$}}(52,23.5)
	\Label\ro{\text {\small$321$}}(52,26.5)
	\Label\ro{\text {\small$321$}}(52,31)
	\Label\ro{\text {\small$\emptyset$}}(58.5,17.5)
	\Label\ro{\text {\small$1$}}(58.5,20.5)
	\Label\ro{\text {\small$2$}}(58.5,23.5)
	\Label\ro{\text {\small$21$}}(58.5,26.5)
	\Label\ro{\text {\small$21$}}(58.5,31)
	\Label\ro{\text {\small$31$}}(63,31)
	\Label\ro{\text {\small$21$}}(63.5,26.5)
	\Label\ro{\text {\small$2$}}(63.5,24.5)
	\Label\ro{\text {\small$21$}}(66.5,23.5)
	\Label\ro{\text {\small$1$}}(66,20.5)
	\Label\ro{\text {\small$\emptyset$}}(66,17.5)
	\hskip14cm
	$$
	\caption{An example of Theorem \ref{Th4.3}.}
	\label{fig:20}
\end{figure}

$\newline$
\subsection{The second main theorem}
$\newline$

Let $T$ be an increasing tableau with shape $\mu/\lambda$. A $rectification$ of $T$ is a straight increasing tableau that can be obtained by applying a sequence of jeu de taquin for $T$. Let $rec(T)$ denote the set of all rectifications of $T$.
There is an obvious fact that for any $\mathcal{T}\in rec(T)$, $\mathcal{T}$ and $T$ are K-jeu de taquin equivalent.

Let $U$ be a standard straight increasing tableau with shape $\lambda$, we denote by $rec_U(T)$ an element of $rec(T)$. It can be obtained by the following steps:

 \begin{itemize}
 	\item Let $k$ be the largest entry in $U$, and let $C_i$ be the set of boxes in $U$ that contain the entry $i$. Setting $T_k:=T$ we recursively define $$T_{i-1}:=jdt_{C_i}(T_i)$$
 	for $i=k,k-1,\dots,1$. Finally, we set $rec_U(T)=T_0$.
 \end{itemize}

\begin{theorem}
	Let $\mathcal{M}$ be a stack polyomino with left justified rows. Given a 01-filling $\pi$ of $\mathcal{M}$ such that there is at most one $1$ in each column, we construct its Hecke growth diagram $\mathcal{M}(\pi)$. Let $c$ be any corner of the right/up boundary of $\mathcal{M}$, labeled by $(\lambda_0,\lambda_1,\dots,\lambda_t)$. Let $\mathcal{R}$ be the largest rectangular region below and to the left of $c$. The word corresponding to the filling in $\mathcal{R}$ is denoted by $w$. Let $P(w)$ be the Hecke insertion tableau of $w$. If the shape of $P(w)$ is $\lambda$, then $\lambda$ and $\lambda_t$ have the same length of the first row and the first column, respectively.
	\label{Th4.11}
\end{theorem}

Before proving Theorem \ref{Th4.11}, we need to prepare the following lemmas that will be used for the proof of this theorem.


\newpage

\begin{lemma}
	Let $P$ be a straight increasing tableau with entries in $[n]$. Let $T$ be the subtableau of $P$ with entries in $[m]$ $(m \leq n)$. Then, there exists $\mathcal{T}\in rec(P/T)$ such that $\mathcal{T}-m=\underbrace{jdt(\cdots jdt}_{m}(P))$, where $\mathcal{T}-m$ is a straight increasing tableau obtained by decreasing all entries of $\mathcal{T}$ by $m$.
	\label{lemma4.5}
\end{lemma}

\begin{proof}
	
	See Fig.\ref{fig:22}, where $T$ is a straight increasing tableau.
	If the top-left box of $T$ is filled with $1$, we perform the jeu de taquin map on $T$. The $\bullet$ is swapped with all the entries of $T$, and finally, $\bullet$ reaches the inner corners of $T$ (see the middle figure in Fig.\ref{fig:22} for an example). Let $C_m$ be the set of these inner corners of $T$. Otherwise, if the top-left box of $T$ is filled with $\alpha$ where $\alpha>1$, we perform the jeu de taquin map on $T$, and let $C_m$ be $\emptyset$.
    
    For $jdt(T)$, which is a straight increasing tableau,
    if the top-left box of $jdt(T)$ is filled with $1$, we perform the jeu de taquin map on $jdt(T)$. The $\bullet$ is swapped with all the entries of $jdt(T)$, and finally, $\bullet$ reaches the inner corners of $jdt(T)$. Let $C_{m-1}$ be the set of these inner corners of $jdt(T)$. 
    Otherwise, if the top-left box of $jdt(T)$ is filled with $\alpha$ where $\alpha>1$, we perform jeu de taquin map on $jdt(T)$, and let $C_{m-1}$ be $\emptyset$.

    We repeat the above process and obtain the sequence $C_1, C_2,\dots, C_m$, where $C_1\cup C_2\cup\dots\cup C_m$ is exactly the shape of $T$. We construct a new straight increasing tableau $U$ with the property that the set of boxes with entries $i$ is  $C_i$ for $1\leq i\leq m$. 
    Let $\mathcal{T}=rec_U(P/T)$. It is easy to check that $\mathcal{T}-m=\underbrace{jdt(\cdots jdt}_{m}(P))$.

\end{proof}

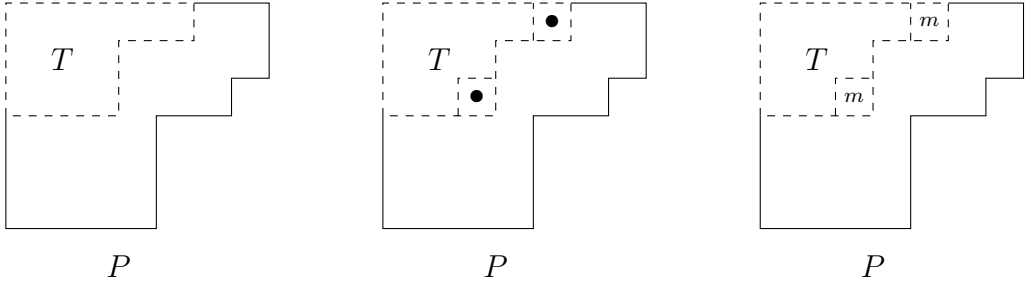
\begin{figure}
	$$
	\begin{tikzpicture}[scale=0.5]
		\node at (1.5,4.5) {$T$};
		\node at (3,-1) {$P$};
		
		\draw (0,3)--(0,0)--(4,0)--(4,3)--(6,3)--(6,4)--(7,4)--(7,6)--(5,6);
		\draw [dashed] 
		(0,3)--(0,6)--(5,6)--(5,5)--(3,5)--(3,3)--(0,3);
	\end{tikzpicture}
	\hbox{\hskip1.5cm}
	\begin{tikzpicture}[scale=0.5]
		\node at (1.5,4.5) {$T$};
		\node at (3,-1) {$P$};
		\node at (2.5,3.5) {$\bullet$};
		\node at (4.5,5.5) {$\bullet$};
		
		\draw (0,3)--(0,0)--(4,0)--(4,3)--(6,3)--(6,4)--(7,4)--(7,6)--(5,6);
		\draw [dashed] 
		(0,3)--(0,6)--(5,6)--(5,5)--(3,5)--(3,3)--(0,3)
		(2,3)--(2,4)--(3,4)
		(4,5)--(4,6);
	\end{tikzpicture}
	\hbox{\hskip1.5cm}
	\begin{tikzpicture}[scale=0.5]
		\node at (1.5,4.5) {$T$};
		\node at (3,-1) {$P$};
		\node at (2.5,3.5) {\tiny $m$};
		\node at (4.5,5.5) {\tiny $m$};
		
		\draw (0,3)--(0,0)--(4,0)--(4,3)--(6,3)--(6,4)--(7,4)--(7,6)--(5,6);
		\draw [dashed] 
		(0,3)--(0,6)--(5,6)--(5,5)--(3,5)--(3,3)--(0,3)
		(2,3)--(2,4)--(3,4)
		(4,5)--(4,6);
	\end{tikzpicture}
	$$
	\caption{An illustration for the proof of Lemma \ref{lemma4.5}.}
	\label{fig:22}
\end{figure}

\begin{lemma}
	Let $w$ be a word with elements contained in $[n]$, and let its Hecke insertion tableau be $P$.
	Let $w^{\prime}$ be the word obtained from $w$ by deleting all elements contained in $[m]$, and let $P^{\prime}$ be the Hecke insertion tableau of $w^{\prime}$. 
	Then $P^{\prime}-m \equiv \underbrace{jdt(\cdots jdt}_{m}(P)).$
	\label{lemma4.6}
\end{lemma}

\begin{proof}
	By Lemma \ref{pro3}, we have $w\equiv row(P)$. 
	Thus, by Lemma \ref{pro1}, 
	$$w\setminus[m]\equiv row(P)\setminus[m]=row(P/[m]),$$
	where $w\setminus[m]$ means deleting 1's, 2's,..., m's from $w$, and $P/[m]$ means deleting the boxes with entries contained in $[m]$ from $P$. Since $$row(P^{\prime})\equiv w'=w\setminus[m],$$
	we have $row(P')\equiv row(P/[m])$. Thus, $P'\equiv P/[m]$. 
	By Theorem \ref{pro2}, $P^{\prime}-m$ and $P/[m]-m$ are $K$-jeu de taquin equivalent. By Lemma \ref{lemma4.5}, we have $P/[m]-m$ is $K$-jeu de taquin equivalent to $\underbrace{jdt(\cdots jdt}_{m}(P))$. Thus, $P^{\prime}-m \equiv \underbrace{jdt(\cdots jdt}_{m}(P)).$
\end{proof}

\begin{lemma}
	Let $T_1$ and $T_2$ be two straight increasing tableaux such that $T_1\equiv T_2$. Then, for any positive integer $k$, $(T_{1}\xleftarrow{\,\mathrm{H}}k)\equiv (T_{2}\xleftarrow{\,\mathrm{H}}k)$.
	\label{lemma4.7}
\end{lemma}

\begin{proof}
	Since $T_1\equiv T_2$, we have $row(T_1)\equiv row(T_2)$. Thus, $row(T_1)\cdot k\equiv row(T_2)\cdot k$, where $row(T_i)\cdot k$ denotes a new word $w_i=(row(T_i),k)$, for $i=1,2$. Let $P_i$ be the Hecke insertion tableau of $w_i$, by Lemma \ref{pro3}, we have $w_i\equiv row(P_i)$. Since $w_1\equiv w_2$, it follows that $row(P_1)\equiv row(P_2)$. Therefore, $P_1\equiv P_2$, i.e. $(T_{1}\xleftarrow{\,\mathrm{H}}k)\equiv (T_{2}\xleftarrow{\,\mathrm{H}}k)$.
\end{proof}

\begin{lemma}
	Let $T_1$ and $T_2$ be two straight increasing tableaux such that $T_1\equiv T_2$. Let $[a,b]$ be an interval of integers. For $i=1,2$, let $T_i|_{[a,b]}$ be the increasing tableau obtained from $T_i$ by deleting all boxes with entries not contained in the interval $[a,b]$. Then $T_1|_{[a,b]}\equiv T_2|_{[a,b]}$.
	\label{lemma4.8}
\end{lemma}

\begin{proof}
	This can be easily proved by using Lemma \ref{pro1}.
\end{proof}

\begin{lemma}
	Let $T_1$ and $T_2$ be two straight increasing tableaux such that $T_1\equiv T_2$. Then $jdt(T_1)\equiv jdt(T_2)$.
	\label{lemma4.9}
\end{lemma}

\begin{proof}
	Since $T_1\equiv T_2$, we have $row(T_1)\equiv row(T_2)$.
	
	If the smallest entry of \(T_i\) is $1$ for \(i = 1, 2\), then \(\mathrm{row}(T_1) \setminus [1] \equiv \mathrm{row}(T_2) \setminus [1]\). Since \(\mathrm{row}(T_i) \setminus [1] = \mathrm{row}(T_i / [1])\), it follows that \(\mathrm{row}(T_1 / [1]) \equiv \mathrm{row}(T_2 / [1])\). Thus, \(T_1 / [1] \equiv T_2 / [1]\). By Theorem \ref{pro2}, \(T_1 / [1]\) and \(T_2 / [1]\) are $K$-jeu de taquin equivalent. Therefore, \(jdt(T_1) \equiv jdt(T_2)\). 
	
	If the smallest entry of \(T_i\) is greater than $1$, then \(jdt(T_i)\) is equivalent to decreasing all entries of \(T_i\) by $1$. Thus, \(\mathrm{row}(T_1 - 1) \equiv \mathrm{row}(T_2 - 1)\), where \(T_i - 1\) means decreasing all entries of \(T_i\) by $1$ for \(i = 1, 2\). Therefore, \(jdt(T_1) \equiv jdt(T_2)\).
\end{proof}

\begin{lemma}
	\cite[Proposition 2.43]{GMPPRST} Let $T_1$ and $T_2$ be two straight increasing tableaux such that $T_1\equiv T_2$. Then $T_1$ and $T_2$ have the same subtableau consisting of the first row and the first column.
	\label{lamma4.10}
\end{lemma}

Now, we give the proof of Theorem \ref{Th4.11}.

\begin{proof}
     We prove this theorem by induction:

	If $c$ is the top-right corner of the first column, let $T_1$ be the straight increasing tableau corresponding to the sequence of partitions along the right border of the first column, and let $P_1$ be the Hecke insertion tableau of the filling in the first column. Then we have $T_1\equiv P_1$. Moreover, $T_1= P_1$. 
	
	Now, let $c$ be the top-right corner of the $i$-th column ($i\geq 1$).
	Let $\mathcal{R}$ be the largest rectangular region below and to the left of $c$. The filling in $\mathcal{R}$ corresponds to the word $w$.
	Let the sequence of partitions along the right border of the $i$-th column corresponds to a straight increasing tableau $T$. 
	Assume that $T \equiv P(w) $.

	\begin{figure}[h]
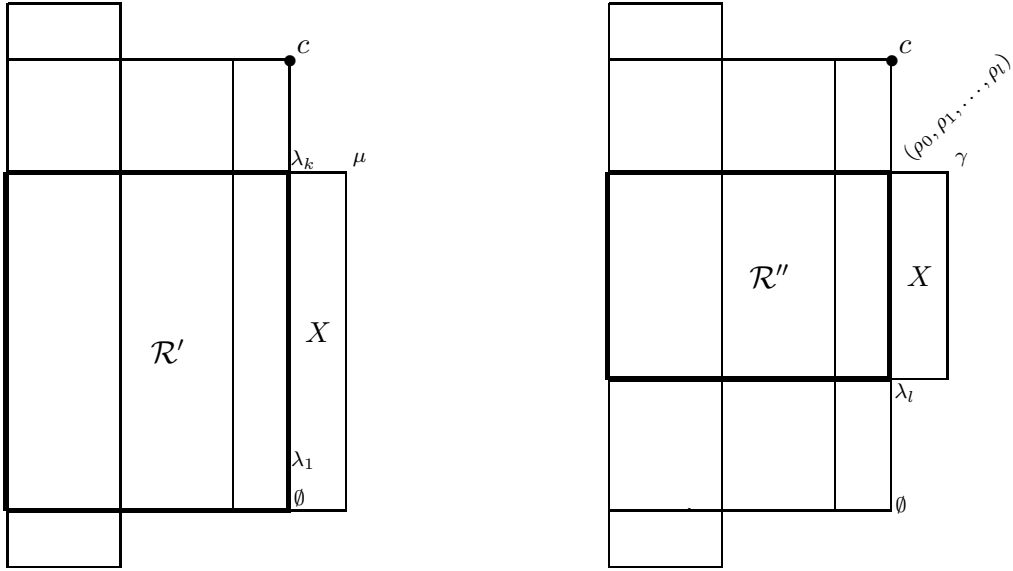

		$$
		\Einheit.25cm
		\PfadDicke{.5pt}
		\Pfad(0,0),111\endPfad
		\Pfad(3,0),111\endPfad
		\Pfad(3,18),111\endPfad
		\Pfad(0,24),111\endPfad
		\Pfad(3,0),222222222222222222222222\endPfad
		\Pfad(6,0),222222222222222222\endPfad
		\Pfad(0,24),555555555555\endPfad
		\Pfad(0,0),555555555555\endPfad
		\Pfad(-12,0),222222222222222222222222\endPfad
		\Pfad(3,18),555555555555555\endPfad
		\Pfad(-12,0),666\endPfad
		\Pfad(-12,24),222\endPfad
		\Pfad(-12,-3),111111\endPfad
		\Pfad(-12,27),111111\endPfad
		\Pfad(-6,27),666\endPfad
		\Pfad(-6,-3),222\endPfad
		\Pfad(-6,0),222222222222222222222222\endPfad
		\Pfad(0,0),222222222222222222222222\endPfad
		\PfadDicke{1.5pt}
		\Pfad(3,18),555555555555555\endPfad
		\Pfad(3,0),555555555555555\endPfad
		\Pfad(3,0),222222222222222222\endPfad
		\Pfad(-12,0),222222222222222222\endPfad
		\Label\ro{\text {\small$X$}}(4,9)
		\Label\ro{\text {\tiny$\lambda_k$}}(3.2,18.2)
		\Label\ro{\text {\tiny$\emptyset$}}(3,0.2)
		\Label\ro{\text {\tiny$\lambda_1$}}(3.2,2.2)
		\Label\ro{\text {\tiny$\mu$}}(6.2,18.2)
		\Label\ro{\text {\small$\bullet$}}(2.5,23.4)
		\Label\ro{\text {\small$c$}}(3.2,24.2)
		\Label\ro{\text {$\mathcal{R'}$}}(-4,8)
		\hbox{\hskip8cm}
		\Einheit.25cm
		\PfadDicke{.5pt}
		\Pfad(0,0),111\endPfad
		\Pfad(3,7),111\endPfad
		\Pfad(3,18),111\endPfad
		\Pfad(0,24),111\endPfad
		\Pfad(3,0),222222222222222222222222\endPfad
		\Pfad(6,7),22222222222\endPfad
		\Pfad(0,24),555555555555\endPfad
		\Pfad(0,0),555555555555\endPfad
		\Pfad(-12,0),222222222222222222222222\endPfad
		\Pfad(3,18),555555555555555\endPfad
		\Pfad(-12,0),666\endPfad
		\Pfad(-12,24),222\endPfad
		\Pfad(-12,-3),111111\endPfad
		\Pfad(-12,27),111111\endPfad
		\Pfad(-6,27),666\endPfad
		\Pfad(-6,-3),222\endPfad
		\Pfad(-6,0),222222222222222222222222\endPfad
		\Pfad(0,0),222222222222222222222222\endPfad
		\PfadDicke{1.5pt}
		\Pfad(3,18),555555555555555\endPfad
		\Pfad(3,7),555555555555555\endPfad
		\Pfad(3,7),22222222222\endPfad
		\Pfad(-12,7),22222222222\endPfad
		\Label\ro{\text {\small$X$}}(4,12)
		\Label\ro{\slanttext {\tiny$(\rho_0,\rho_1,\dots,\rho_l)$}}(6,21)
		\Label\ro{\text {\tiny$\gamma$}}(6.2,18.2)
		\Label\ro{\text {\tiny$\emptyset$}}(3,-0.2)
		\Label\ro{\text {\tiny$\lambda_l$}}(3.2,5.8)
		\Label\ro{\text {\small$\bullet$}}(2.5,23.4)
		\Label\ro{\text {\small$c$}}(3.2,24.2)
		\Label\ro{\text {$\mathcal{R''}$}}(-4,12)
		\hskip-2cm
		.$$
		\caption{An illustration of the $i$-th column and the ($i+1$)-th column.}
		\label{fig:columns}
	\end{figure}

	(1) If the $(i+1)$-th column is bottom justified with the $i$-th column (see the left side of Fig.\ref{fig:columns}), let the top-left corner of the $(i+1)$-th column be labeled $\lambda_k$. Let the sequence of partitions $(\emptyset,\lambda_1,\dots,\lambda_k)$ along the left border of the $(i+1)$-th column corresponds to the straight increasing tableau $T'$, and let the filling in $\mathcal{R'}$ corresponds to the word $w'$. Since $P(w)\equiv T$, by Lemma \ref{lemma4.8}, we have $$P(w')\equiv T'.$$ 
	
	Suppose that the filling of the {\it $(i+1)$-th} column is at the $r$-th row (from the bottom of the $(i+1)$-th column). Then the sequence of partitions along the right border of the $(i+1)$-th column corresponds to a straight increasing tableau $$U=(T'\xleftarrow{\,\mathrm{H}}r).$$
	
	On the other hand, the filling in the rectangular region below and to the left of the top-right corner of the $(i+1)$-th column corresponds to the word $u=w'\cdot r$ (where $u$ is formed by appending $r$ to the end of $w'$), and $$P(u)=P(w')\xleftarrow{\,\mathrm{H}}r.$$
	
	Therefore, by Lemma \ref{lemma4.7}, $$U\equiv P(u).$$
	
	Let the shape of $U$ is $\mu$, and let the shape of $P(u)$ is $\lambda$. By Lemma \ref{lamma4.10}, $\lambda$ and $\mu$ have the same length of the first row and the first column, respectively.

	(2) If the $(i+1)$-th column is not bottom justified with the $i$-th column (see the right side of Fig.\ref{fig:columns}), let the top-left corner of the $(i+1)$-th column be labeled by $(\rho_0,\rho_1,\dots,\rho_l)$. By the discuss of (1), $P(w')\equiv T'$. Therefore, by Lemma \ref{lemma4.9}, 
	$$\underbrace{jdt(\cdots jdt}_{l}(P(w')))\equiv \underbrace{jdt(\cdots jdt}_{l}(T')).$$
	
	Let the fillings in $\mathcal{R''}$ corresponds to the word $w''$. 
	By Lemma \ref{lemma4.6},
	$$\underbrace{jdt(\cdots jdt}_{l}(P(w')))\equiv P(w'').$$
	
	Thus, we have
	$$\underbrace{jdt(\cdots jdt}_{l}(T'))\equiv P(w'').$$
	
	Suppose that the filling of the $(i+1)$-th column is at the $s$-th row (from the bottom of the $(i+1)$-th column). Then the sequence of partitions along the right border of the $(i+1)$-th column corresponds to a straight increasing tableau
	$$V=\underbrace{jdt(\cdots jdt}_{l}(T'))\xleftarrow{\,\mathrm{H}}s.$$
	
	On the other hand, the filling in the rectangular region below and to the left of the top-right corner of the $(i+1)$-th corresponds to the word $v=w''\cdot s$, and $$P(v)=P(w'')\xleftarrow{\,\mathrm{H}}s.$$
	
	Therefore, by Lemma \ref{lemma4.7}, $$V\equiv P(v).$$
	
	Let the shape of $V$ be $\gamma$, and let the shape of $P(v)$ be $\lambda$. By Lemma \ref{lamma4.10}, $\lambda$ and $\gamma$ have the same length of the first row and the first column, respectively.
\end{proof}

Readers can see Fig.\ref{Hecke growth diagram} for an example. 

\begin{remark}
	For the proof of Theorem \ref{Th4.11}, we could not obtain $P(w)=T$ for every corner. Refer to the left side of Fig.\ref{fig:50}: when we complete this Hecke growth diagram, $c$ is labeled with the partition $\lambda=(3,2,1)$. However, the filling in the rectangular region below and to the left of $c$ has the Hecke insertion tableau as shown on the right side of Fig.\ref{fig:50}.
\end{remark}

	\begin{figure}[h]
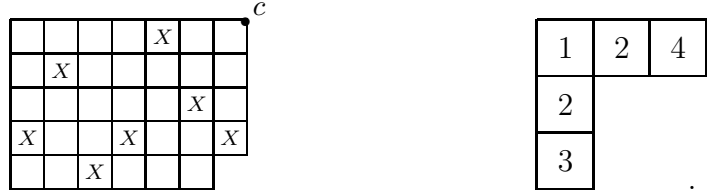

		$$
		\Einheit.15cm
		\PfadDicke{.5pt}
		\Pfad(0,0),111111111111111111\endPfad
		\Pfad(0,3),111111111111111111111\endPfad
		\Pfad(0,6),111111111111111111111\endPfad
		\Pfad(0,9),111111111111111111111\endPfad
		\Pfad(0,12),111111111111111111111\endPfad
		\Pfad(0,15),111111111111111111111\endPfad
		\Pfad(0,0),222222222222222\endPfad
		\Pfad(3,0),222222222222222\endPfad
		\Pfad(6,0),222222222222222\endPfad
		\Pfad(9,0),222222222222222\endPfad
		\Pfad(12,0),222222222222222\endPfad
		\Pfad(15,0),222222222222222\endPfad
		\Pfad(18,0),222222222222222\endPfad
		\Pfad(21,3),222222222222\endPfad
		\Label\ro{\text {\tiny$X$}}(1,4)
		\Label\ro{\text {\tiny$X$}}(4,10)
		\Label\ro{\text {\tiny$X$}}(7,1)
		\Label\ro{\text {\tiny$X$}}(10,4)
		\Label\ro{\text {\tiny$X$}}(13,13)
		\Label\ro{\text {\tiny$X$}}(16,7)
		\Label\ro{\text {\tiny$X$}}(19,4)
		\Label\ro{\text {\tiny$\bullet$}}(20.3,14.3)
		\Label\ro{\text {\small$c$}}(21.5,15.5)
		\hbox{\hskip7cm}
		\Einheit.25cm
		\PfadDicke{.5pt}
		\Pfad(0,0),222222222\endPfad
		\Pfad(3,0),222222222\endPfad
		\Pfad(6,6),222\endPfad
		\Pfad(9,6),222\endPfad
		\Pfad(0,0),111\endPfad
		\Pfad(0,3),111\endPfad
		\Pfad(0,6),111111111\endPfad
		\Pfad(0,9),111111111\endPfad
		\Label\ro{1}(1,7)
		\Label\ro{2}(1,4)
		\Label\ro{2}(4,7)
		\Label\ro{3}(1,1)
		\Label\ro{4}(7,7)
		\hskip2cm
		.$$
		\caption{An example for $P(w)\neq T$.}
		\label{fig:50}
	\end{figure}

Based on Theorem \ref{Th4.11} and Theorem \ref{Th2.1}, we have the following.

\begin{corollary}
	Let $\mathcal{M}$ be a stack polyomino. Given a 01-filling $\pi$ of $\mathcal{M}$ such that there is at most one $1$ in each column. The labellings of corners along the right-top border of the Hecke growth diagram $\mathcal{M}(\pi)$ determine the lengths of the longest increasing and decreasing chains of the largest rectangular region below and to the left of the corners.
	\label{corollary4.11}
\end{corollary}

\section{Application of these two main theorems}

In this section, we use Theorem \ref{Th4.3} and Theorem \ref{Th4.11} to prove Corollary 2 of \cite{GP2020} and Theorem 4.2 of \cite{CGP}, and we explain how our Hecke growth diagram construction generalizes the growth diagram techniques of Rubey for stack polyominoes with at most one $1$ per column and row.

\subsection{An alternative proof of Corollary 2 from \cite{GP2020}}
$\newline$


Let $\mathcal{M}$ be a stack polyomino with left-justified rows. Let $\pi$ be a $01$-filling of $\mathcal{M}$ such that there is at most one $1$ in each column. A sequence of $1$'s in the filling $\pi$ is a $chain$ if the smallest rectangle containing all the $1$'s in the sequence is completely contained in $\mathcal{M}$. A chain is called an {\it ne-chain} if each 1 in the chain is strictly above and strictly to the right of all preceding $1$'s in the chain. Similarly, we define {\it se-chain}. The $length$ of a chain is the number of $1$'s in the chain.
Let $N(\mathcal{M};n;ne=u,se=v)$ denote the number of $01$-fillings of the stack polyomino $\mathcal{M}$ consisting of $n$ $1$'s, where the longest ne-chain is of length $u$ and the longest se-chain is of length $v$.

\begin{corollary}
	\cite[Corollary 2]{GP2020} Let $\mathcal{M}$ be a stack polyomino. Then 
	$$N(\mathcal{M};n;ne=u,se=v)=N(\mathcal{M};n;ne=v,se=u).$$
\end{corollary}

\begin{proof}
We define a bijection between the $01$-fillings counted by $N(\mathcal{M};n;ne=u,se=v)$ and those counted by $N(\mathcal{M};n;ne=v,se=u)$. Consider a $01$-filling $\pi$ counted by
$N(\mathcal{M};n;ne=u,se=v)$. From $\pi$, we derive the sequence $(\emptyset=l_0,l_1,\dots,l_{n}=\emptyset)$ along the right/up boundary of $\mathcal{M}$ by the Hecke growth diagram $\mathcal{M}(\pi)$. 
Next, we apply the inverse mapping of the proof of Theorem \ref{Th4.3} to the sequence $(\emptyset=l_0^{\prime},l_1^{\prime},\dots,l_{n}^{\prime}=\emptyset)$ of conjugate partitions, where $l_j^{\prime}$ is a conjugate partition of $l_j$ 
for $0\leq j\leq n$. Note that for partition $\lambda^r$, we define its conjugation by conjugating $\lambda$ while preserving the position of terminate box, an example is shown in Fig.\ref{fig:51}. 
Thus, we obtain a $01$-filling $\pi^{\prime}$ counted by $N(\mathcal{M};n;ne=v,se=u)$. See Fig.\ref{fig:23} for an example. 

\end{proof}

\begin{figure}[h]
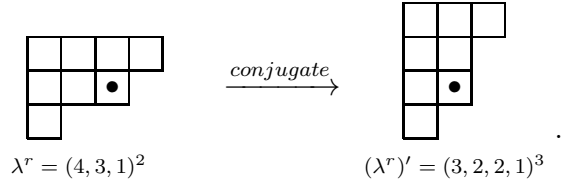

	$$
	\Einheit.15cm
	\PfadDicke{.5pt}
	\Pfad(12,6),222\endPfad
	\Pfad(9,3),222222\endPfad
	\Pfad(6,3),222222\endPfad
	\Pfad(3,0),222222222\endPfad
	\Pfad(0,0),222222222\endPfad
	\Pfad(0,9),111111111111\endPfad
	\Pfad(0,6),111111111111\endPfad
	\Pfad(0,3),111111111\endPfad
	\Pfad(0,0),111\endPfad
	\Label\ro{\text {\small$\bullet$}}(7,4)
	\Label\ro{\text {\tiny$\lambda^r=(4,3,1)^2$}}(4,-3)
	\hbox{\hskip3cm}
	\Label\ro{\xrightarrow{\text{$conjugate$}}}(2,5)
	\hbox{\hskip2cm}
	\Einheit.15cm
	\PfadDicke{.5pt}
	\Pfad(9,9),222\endPfad
	\Pfad(6,3),222222222\endPfad
	\Pfad(3,0),222222222222\endPfad
	\Pfad(0,0),222222222222\endPfad
	\Pfad(0,12),111111111\endPfad
	\Pfad(0,9),111111111\endPfad
	\Pfad(0,6),111111\endPfad
	\Pfad(0,3),111111\endPfad
	\Pfad(0,0),111\endPfad
	\Label\ro{\text {\small$\bullet$}}(4,4)
	\Label\ro{\text {\tiny$(\lambda^r)^{\prime}=(3,2,2,1)^3$}}(4,-3)
	\hskip2cm
	.$$
	\caption{The conjugate of $(4,3,1)^2$.}
	\label{fig:51}
\end{figure}

\begin{figure}
	$$
	\begin{tikzpicture}[scale=1.2]
		\node (00) at (0,0) {$\emptyset$};
		\node (10) at (1,0) {$\emptyset$};
		\node (20) at (2,0) {$\emptyset$};
		\node (30) at (3,0) {$\emptyset$};
		\node (01) at (0,1) {$\emptyset$};
		\node (11) at (1,1) {$1$};
		\node (21) at (2,1) {$1$};
		\node (31) at (3,1) {$1$};
		\node (02) at (0,2) {$\emptyset$};
		\node (12) at (1,2) {$1$};
		\node (22) at (2,2) {$1$};
		\node (32) at (3,2) {$1$};
		\node (42) at (4,2) {$\emptyset$};
		\node (52) at (5,2) {$\emptyset$};
		\node (62) at (6,2) {$\emptyset$};
		\node (03) at (0,3) {$\emptyset$};
		\node (13) at (1,3) {$1$};
		\node (23) at (2,3) {$1$};
		\node (33) at (3,3) {$1$};
		\node (43) at (4,3) {$\emptyset$};
		\node (53) at (5,3) {$\emptyset$};
		\node (63) at (6,3) {$\emptyset$};
		\node (73) at (7,3) {$\emptyset$};
		\node (04) at (0,4) {$\emptyset$};
		\node (14) at (1,4) {$1$};
		\node (24) at (2,4) {$1$};
		\node (34) at (3,4) {$1$};
		\node (44) at (4,4) {$\emptyset$};
		\node (54) at (5,4) {$\emptyset$};
		\node (64) at (6,4) {$1$};
		\node (74) at (7,4) {$1$};
		\node (84) at (8,4) {$\emptyset$};
		
		\node (05) at (0,5) {$\emptyset$};
		\node (15) at (1,5) {$1$};
		\node (25) at (2,5) {$1$};
		\node (35) at (3,5) {$1$};
		\node (45) at (4,5) {$\emptyset$};
		\node (55) at (5,5) {$\emptyset$};
		\node (65) at (6,5) {$1$};
		\node (75) at (7,5) {$1$};
		\node (85) at (8,5) {$1$};
		
		\node (06) at (0,6) {$\emptyset$};
		\node (16) at (1,6) {$1$};
		\node (26) at (2,6) {$1$};
		\node (36) at (3,6) {$1$};
		\node (46) at (4,6) {$1$};
		\node (56) at (5,6) {$1$};
		\node (66) at (6,6) {\tiny $(11,11)$};
		\node (76) at (7,6) {\tiny $(1,21)$};
		\node (86) at (8,6) {$11$};
		
		\node (07) at (0,7) {$\emptyset$};
		\node (17) at (1,7) {$1$};
		\node (27) at (2,7) {$1$};
		\node (37) at (3,7) {\tiny $(1,1,2)$};
		\node (47) at (4,7) {$11$};
		\node (57) at (5,7) {$21$};
		
		\node (08) at (0,8) {$\emptyset$};
		\node (18) at (1,8) {$1$};
		\node (28) at (2,8) {$2$};
		\node (38) at (3,8) {$21$};
		
		\node at (0.5,0.5) {$X$};
		\node at (1.5,7.5) {$X$};
		\node at (2.5,6.5) {$X$};
		\node at (3.5,5.5) {$X$};
		\node at (4.5,6.5) {$X$};
		\node at (5.5,3.5) {$X$};
		\node at (6.5,5.5) {$X$};
		\node at (7.5,4.5) {$X$};
		

		\draw (00)--(10)--(20)--(30)
		(01)--(11)--(21)--(31)
		(02)--(12)--(22)--(32)--(42)--(52)--(62)
		(03)--(13)--(23)--(33)--(43)--(53)--(63)--(73)
		(04)--(14)--(24)--(34)--(44)--(54)--(64)--(74)--(84)
		(05)--(15)--(25)--(35)--(45)--(55)--(65)--(75)--(85)
		(06)--(16)--(26)--(36)--(46)--(56)--(66)--(76)--(86)
		(07)--(17)--(27)--(37)--(47)--(57)
		(08)--(18)--(28)--(38)
		
		(00)--(01)--(02)--(03)--(04)--(05)--(06)--(07)--(08)
		(10)--(11)--(12)--(13)--(14)--(15)--(16)--(17)--(18)
		(20)--(21)--(22)--(23)--(24)--(25)--(26)--(27)--(28)
		(30)--(31)--(32)--(33)--(34)--(35)--(36)--(37)--(38)
		(42)--(43)--(44)--(45)--(46)--(47)
		(52)--(53)--(54)--(55)--(56)--(57)
		(62)--(63)--(64)--(65)--(66)
		(73)--(74)--(75)--(76)
		(84)--(85)--(86);

	\end{tikzpicture}
	$$
	\caption{An example of constructing $\pi^\prime$ from $\pi$ of Fig.\ref{Hecke growth diagram}.}
	\label{fig:23}
\end{figure}
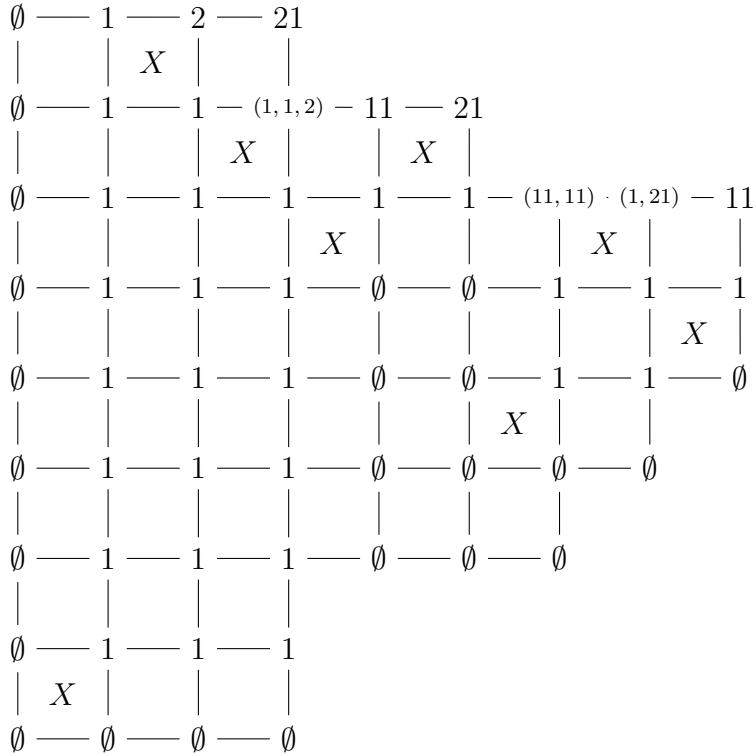

\subsection{An alternative proof of Theorem 4.2 from \cite{CGP}}
$\newline$

A {\it linked partition} of [$n$] is a collection of nonempty subsets $B_1,B_2,\dots,B_k$ of [$n$], called {\it blocks}, such that the union of $B_1,B_2,\dots,B_k$ is [$n$] and any two distinct blocks are nearly disjoint. Two distinct blocks $B_i$ and $B_j$ are said to be nearly disjoint if for any $t\in B_i\cap B_j$, one of the following conditions holds:

\begin{itemize}
	\item [(1)] $t$=min($B_i$), $|B_i|>1$, and $t\neq$min($B_j$),
	\item [(2)] $t$=min($B_j$), $|B_j|>1$, and $t\neq$min($B_i$).
\end{itemize}

Given a linked partition $P$ of [$n$], a block $\{i_1,i_2,\dots,i_m\}$, $i_1<i_2<\cdots<i_m$, of $P$ is represented by the set of pairs $\{(i_1,i_2),(i_1,i_3),\dots,(i_1,i_m)\}$. More generally, a linked partition is represented by the union of all set of pairs, the union being taken over all its blocks. This representation is called the {\it standard representation} of the linked partition. For example, the linked partition $\{\{1,3,6\},\{2,5,8\},\{4\},\{5,9\},\{6,7\}\}$ is represented as the set $\{(1,3),(1,6),(2,5),(2,8),(5,9),(6,7)\}$. Conversely, we can  uniquely obtain the linked partition from its standard representation.

Let $P$ be a linked partition of [$n$], a subset $\{(i_1,j_1),(i_2,j_2),\dots,(i_k,j_k)\}$ of its standard representation form a {\it $k$-crossing} if
$$i_1<i_2<\cdots<i_k<j_1<j_2<\cdots<j_k.$$
Similarly, we say that $\{(i_1,j_1),(i_2,j_2),\dots,(i_k,j_k)\}$ form a {\it $k$-nesting} if
$$i_1<i_2<\cdots<i_k<j_k<\cdots<j_2<j_1.$$
We write $cross(P)$ for the maximal number $k$ such that $P$ has a $k$-crossing, and we write $nest(P)$ for the maximal number $k$ such that $P$ has a $k$-nesting. 

For the linked partition $P$, let $compl(P)$ be the set of the left components of the pairs of its standard representation, and let $compr(P)$ be the set of the right components of the pairs of its standard representation. For example, let $P=\{\{1,3,6\},\{2,5,8\},\{4\},\{5,9\},\{6,7\}\}$, then $compl(P)=\{1,2,5,6\}$, and $compr(P)=\{3,5,6,7,8,9\}$.

Based on the definitions above, we recall \cite[Theorem 4.2]{CGP} as follows.

\begin{theorem}\cite[Theorem 4.2]{CGP}
	Let $n,x,y$ be the positive integers, and let $S$ and $T$ be two subset of $[n]$. Then the number of linked partitions $P$ of $[n]$ with $cross(P)=x$, $nest(P)=y$, $compl(P)=S$, $compr(P)=T$ is equal to the number of linked partitions of $[n]$ with $cross(P)=y$, $nest(P)=x$, $compl(P)=S$, $compr(P)=T$.
	\label{Th4.15}
\end{theorem}

In \cite{GP2020}, Guo and Poznanovi\'c established a bijection between linked partitions on [$n$] and fillings of $\Delta_n$, the triangular shape with $n-1$ cells in the bottom row, $n-2$ cells in the row above, etc., and 1 cell in the top-most row. See Fig.\ref{fig:24} for an example in which $n=7$. (The filling and labeling of the corners should be ignored at this point. For convenience, we also joined pending edges at the right and at the top of $\Delta_n$.) We represent a linked partition $P$ of [$n$], given by its standard representation, as a filling, by putting an $X$ to the box in row $i$ (from bottom to top) and column $j$ (counted from right to left, including one empty column) if and only if $(i,j)$ is a pair in the standard representation of $P$. See Fig.\ref{fig:24}, this defines a correspondence between linked partition $\{\{1,2,3,5,6\},\{2,4,7\}\}$ and 01-filling of $\Delta_7$ with at most one 1 in each column. Moreover, a $k$-crossing of $P$ corresponds to a $se$-chain of length $k$ in $\Delta_n$, and a $k$-nesting of $P$ corresponds to a $ne$-chain of length $k$ in $\Delta_n$.

\begin{figure}
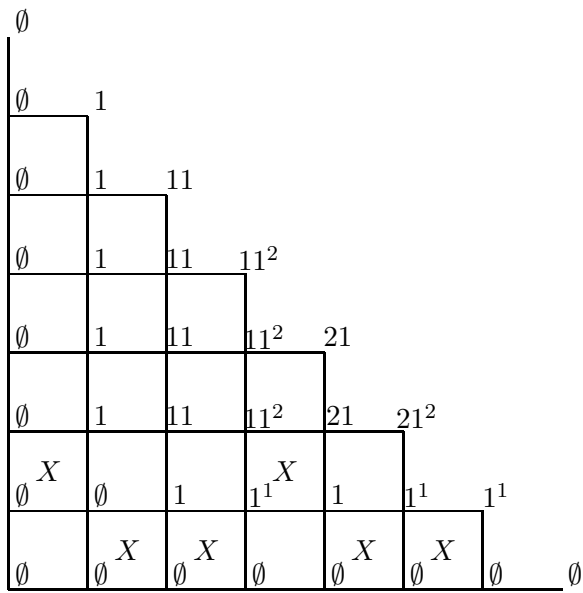

	$$
	\Einheit.35cm
	\PfadDicke{.5pt}
	\Pfad(0,0),222222222222222222222\endPfad
	\Pfad(3,0),222222222222222222\endPfad
	\Pfad(6,0),222222222222222\endPfad
	\Pfad(9,0),222222222222\endPfad
	\Pfad(12,0),222222222\endPfad
	\Pfad(15,0),222222\endPfad
	\Pfad(18,0),222\endPfad
	\Pfad(0,0),111111111111111111111\endPfad
	\Pfad(0,3),111111111111111111\endPfad
	\Pfad(0,6),111111111111111\endPfad
	\Pfad(0,9),111111111111\endPfad
	\Pfad(0,12),111111111\endPfad
	\Pfad(0,15),111111\endPfad
	\Pfad(0,18),111\endPfad
	\Label\ro{\text {\small$\emptyset$}}(0,0.1)
	\Label\ro{\text {\small$\emptyset$}}(0,3.1)
	\Label\ro{\text {\small$\emptyset$}}(0,6.1)
	\Label\ro{\text {\small$\emptyset$}}(0,9.1)
	\Label\ro{\text {\small$\emptyset$}}(0,12.1)
	\Label\ro{\text {\small$\emptyset$}}(0,15.1)
	\Label\ro{\text {\small$\emptyset$}}(0,18.1)
	\Label\ro{\text {\small$\emptyset$}}(0,21.1)
	\Label\ro{\text {\small$\emptyset$}}(3,0.1)
	\Label\ro{\text {\small$\emptyset$}}(3,3.1)
	\Label\ro{\text {\small$1$}}(3,6.1)
	\Label\ro{\text {\small$1$}}(3,9.1)
	\Label\ro{\text {\small$1$}}(3,12.1)
	\Label\ro{\text {\small$1$}}(3,15.1)
	\Label\ro{\text {\small$1$}}(3,18.1)
	\Label\ro{\text {\small$\emptyset$}}(6,0.1)
	\Label\ro{\text {\small$1$}}(6,3.1)
	\Label\ro{\text {\small$11$}}(6,6.1)
	\Label\ro{\text {\small$11$}}(6,9.1)
	\Label\ro{\text {\small$11$}}(6,12.1)
	\Label\ro{\text {\small$11$}}(6,15.1)
	\Label\ro{\text {\small$\emptyset$}}(9,0.1)
	\Label\ro{\text {\small$1^1$}}(9.1,3.1)
	\Label\ro{\text {\small$11^2$}}(9.2,6.1)
	\Label\ro{\text {\small$11^2$}}(9.2,9.1)
	\Label\ro{\text {\small$11^2$}}(9,12.1)
	\Label\ro{\text {\small$\emptyset$}}(12,0.1)
	\Label\ro{\text {\small$1$}}(12,3.1)
	\Label\ro{\text {\small$21$}}(12.1,6.1)
	\Label\ro{\text {\small$21$}}(12,9.1)
	\Label\ro{\text {\small$\emptyset$}}(15,0.1)
	\Label\ro{\text {\small$1^1$}}(15,3.1)
	\Label\ro{\text {\small$21^2$}}(15,6.1)
	\Label\ro{\text {\small$\emptyset$}}(18,0.1)
	\Label\ro{\text {\small$1^1$}}(18,3.1)
	\Label\ro{\text {\small$\emptyset$}}(21,0.1)
	\Label\ro{\text {\small$X$}}(1,4)
	\Label\ro{\text {\small$X$}}(4,1)
	\Label\ro{\text {\small$X$}}(7,1)
	\Label\ro{\text {\small$X$}}(10,4)
	\Label\ro{\text {\small$X$}}(13,1)
	\Label\ro{\text {\small$X$}}(16,1)
	\hskip6cm
	$$
	\caption{The linked partition $\{\{1,2,3,5,6\},\{2,4,7\}\}$ and the corresponding $01$-filling.}
	\label{fig:24}
\end{figure}

\begin{proof}[Proof of Theorem \ref{Th4.15}]
	Given a linked partition $P$ of [$n$] such that $cross(P)=x$, $nest(P)=y$, $compl(P)=S$, $compr(P)=T$, we associate $P$ with a 01-filling $\pi$ of $\Delta_n$. By the definition of nearly disjoint, each column of the 01-filling $\pi$ contains at most one 1 in each column. We then construct the Hecke growth diagram of $\pi$. Next, we conjugate all partitions along the top-right border of $\Delta_n$ and reconstruct the filling of $\Delta_n$. We denote the new filling as $\omega$, which corresponds to a new linked partition $P'$. It is easy to observe that $cross(P')=y$ and $nest(P')=x$. Since $\omega$ has $X$ in row $i$ if and only if $\pi$ has $X$ in row $i$, and $\omega$ has $X$ in column $j$ if and only if $\pi$ has $X$ in column $j$, we have
	$compl(P')=S$, $compr(P')=T$. Therefore, we conclude the proof of Theorem \ref{Th4.15}.
\end{proof}

\subsection{Generalization of growth diagram for stack polyomino in \cite{Rubey}.}
$\newline$

In \cite{Rubey}, Rubey consider the $01$-fillings of the stack polyomino $\mathcal{M}$ with at most one $1$ in each row and column, and provides a bijection between the set of fillings of $\mathcal{M}$ and sequences of partitions along right/up boundary of $\mathcal{M}$. (cf. \cite[Proposition 7.4]{Rubey}.) In this paper, Theorem \ref{Th4.3} extends the bijection to words.

Since the Hecke insertion algorithm is a generalization of the RSK algorithm, for a permutation $\pi$, its Hecke insertion tableau is identical to its RSK insertion tableau.
Furthermore, if $\pi$ is a permutation, the following proposition holds:

\begin{proposition}
	Let $\omega$ be the permutation obtained from $\pi$ by deleting $1$ and decreasing all remaining elements by $1$. Then $P(\omega)=jdt(P(\pi))$.
\end{proposition}

This proposition is a consequence of \cite[Corollary A.1.2.6]{Stanley}, as noted in the proof of \cite[A.1.2.10]{Stanley}. Therefore, for Theorem \ref{Th4.11}, if $\pi$ corresponds to a permutation, then $\lambda=\lambda_t$.

\end{document}